\documentclass[10pt]{amsart}

\usepackage[utf8]{inputenc} 
\usepackage{amssymb,amscd,graphicx}
 \usepackage{color}
 \usepackage[normalem]{ ulem }
 \usepackage{soul}
 \usepackage{xcolor}
 \usepackage{hyperref}

 \usepackage{enumerate}

 \usepackage{mathrsfs}
 \usepackage{multirow}

\usepackage[foot]{amsaddr}
\begin{document}
\newtheorem{lemma}{Lemma}[section]
\newtheorem{proposition}{Proposition}[section]
\newtheorem{corollary}{Corollary}[section]
\numberwithin{equation}{section}
\newtheorem{theorem}{Theorem}[section]
\theoremstyle{remark}
\newtheorem{example}{Example}[section]
\newtheorem*{ack}{Acknowledgment}
\theoremstyle{definition}
\newtheorem{definition}{Definition}[section]
\theoremstyle{remark}
\newtheorem*{notation}{Notation}
\theoremstyle{remark}
\newtheorem{remark}{Remark}[section]
\newenvironment{Abstract}
{\begin{center}\textbf{\footnotesize{Abstract}}%
\end{center} \begin{quote}\begin{footnotesize}}
{\end{footnotesize}\end{quote}\bigskip}
\newenvironment{nome}
{\begin{center}\textbf{{}}%
\end{center} \begin{quote}\end{quote}\bigskip}

\newcommand{\triple}[1]{{|\!|\!|#1|\!|\!|}}
\newcommand{\xx}{\langle x\rangle}
\newcommand{\ep}{\varepsilon}
\newcommand{\al}{\alpha}
\newcommand{\be}{\beta}
\newcommand{\de}{\partial}
\newcommand{\la}{\lambda}
\newcommand{\La}{\Lambda}
\newcommand{\ga}{\gamma}
\newcommand{\del}{\delta}
\newcommand{\Del}{\Delta}
\newcommand{\sig}{\sigma}
\newcommand{\ome}{\Omega^n}
\newcommand{\Ome}{\Omega^n}
\newcommand{\C}{{\mathbb C}}
\newcommand{\N}{{\mathbb N}}
\newcommand{\Z}{{\mathbb Z}}
\newcommand{\R}{{\mathbb R}}
\newcommand{\T}{{\mathbb T}}
\newcommand{\Rn}{{\mathbb R}^{n}}
\newcommand{\Rnu}{{\mathbb R}^{n+1}_{+}}
\newcommand{\Cn}{{\mathbb C}^{n}}
\newcommand{\spt}{\,\mathrm{supp}\,}
\newcommand{\Lin}{\mathcal{L}}
\newcommand{\SSS}{\mathcal{S}}
\newcommand{\F}{\mathcal{F}}
\newcommand{\xxi}{\langle\xi\rangle}
\newcommand{\eei}{\langle\eta\rangle}
\newcommand{\xei}{\langle\xi-\eta\rangle}
\newcommand{\yy}{\langle y\rangle}
\newcommand{\dint}{\int\!\!\int}
\newcommand{\hatp}{\widehat\psi}
\renewcommand{\Re}{\;\mathrm{Re}\;}
\renewcommand{\Im}{\;\mathrm{Im}\;}

\title[
New bounds on the high Sobolev norms of the 1d NLS solutions
]
{
New bounds on the high Sobolev norms of the 1d NLS solutions
}
\author{Diego Berti$^{1}$} 
\address{$^{1}$Dipartimento di Matematica, Universit\`a di Torino, Via Carlo Alberto 10, 10123, Torino, Italy}
\email{diego.berti@unito.it}
\author{Fabrice Planchon$^{2,3}$}
\address{$^{2}$Sorbonne Universit\'e, CNRS, IMJ-PRG F-75005 Paris, France}
\address{$^{3}$Institut universitaire de France (IUF), France}
\email{fabrice.planchon@sorbonne-universite.fr}
\author{Nikolay Tzvetkov$^{3,4}$}
\address{$^{4}$
Ecole Normale Sup\'erieure de Lyon, UMPA, UMR CNRS-ENSL 5669, 46, all\'ee d'Italie, 69364-Lyon Cedex 07, France
}
\email{nikolay.tzvetkov@ens-lyon.fr}
\author{Nicola Visciglia$^{5}$}
\address{$^{5}$Dipartimento di Matematica, Universit\`a di Pisa, Largo B. Pontevorvo 5, 56124, Pisa, Italy}
\email{nicola.visciglia@unipi.it}

\thanks{The first and fourth authors are supported by the PRIN project 2020XB3EFL and GNAMPA of INDAM, the first author acknowledges financial support also by PNRR MUR - M4C2 1.5 grant ECS00000036, the second author was partially supported by ERC grant ANADEL n. 757996, the third author is partially supported by the ANR project Smooth ANR-22-CE40-0017 }

 \begin{abstract}
 We introduce modified energies that are suitable to get upper bounds on the high Sobolev norms for solutions to the $1$D periodic NLS.
  Our strategy is rather flexible and allows us to get a new and simpler proof of the bounds obtained by Bourgain in the case of the quintic nonlinearity,
  as well as its extension 
  to the case of higher order nonlinearities. Our main ingredients are a combination of integration by parts and classical dispersive estimates. 
  \end{abstract}

 \maketitle
 \par \noindent

\section{Introduction}
This work fits in the line of research initiated by Bourgain in \cite{B} aiming to prove polynomial upper bounds on the solutions of nonlinear dispersive PDE's on compact spatial domains. Such bounds necessarily rely on the dispersive smoothing effects satisfied by the solutions.  
Starting from the pioneering paper \cite{B} many contributions appeared in the literature about the topic, we mention in particular 
\cite{CKO}, \cite{PTV_first}, \cite{PTV2}, \cite{PTVHarm}, \cite{So1}, \cite{So2}, \cite{So3}, \cite{S}, \cite{HT}, \cite{JT}, \cite{Z} 
and the references therein.
Also the question of lower bounds on the Sobolev norms, namely the existence of unbounded orbits in Sobolev spaces, has been extensively studied in the literature, see \cite{CKSTT2}, \cite{G}, \cite{GHP}, \cite{GK}, \cite{HPTV}, \cite{HP} and the references therein.
\\

We  consider from now on the non-linear Schr\" odinger equation in the one-dimensional periodic setting
\begin{equation}
\label{e:T1-NLS}
\begin{cases}
i\partial_t u + \partial_x^2 u - \, u |u|^{2p}=0, \quad (t,x)\in \mathbb{R} \times \mathbb{T},\\ u(0,x)=u_0(x)\in H^k(\mathbb T),
\end{cases}
\end{equation}
where $p,k\in \mathbb N$ and $p\geq 2$, $k\geq 1$. It is well--known that the Cauchy problem \eqref{e:T1-NLS}
admits one unique global solution in the space ${\mathcal C}(\R; H^k(\T))$ (see Section \ref{prel} for details about the Cauchy theory).
Hence the rest of the paper is devoted to the analysis of the behavior of such solutions for large times, in particular we consider the question
of a--priori bounds of the possible growth of Sobolev norms.\\

It is worth mentioning that if we consider the same Cauchy problem
posed on the line $\R$ then no growth phenomenon of Sobolev norms is possible, indeed all the Sobolev norms are uniformly bounded as a consequence of the  scattering results proved in \cite{N} for $p>2$ and \cite{D} for $p=2$.
We also recall that growth of Sobolev norms is not allowed for solutions to \eqref{e:T1-NLS} for $p=1$. Indeed cubic NLS  (in both cases $\T$ and $\R$) is completely integrable and in particular the corresponding solutions satisfy infinitely many conservation laws that allow to control uniformly in time
the Sobolev norms $H^k(\T)$ for every $k\in \N$.
\\

It is quite elementary to show that the Cauchy problem \eqref{e:T1-NLS} is globally well-posed for $k\geq 1$ (see Section \ref{prel} below for a precise statement on well-posedness).
In the sequel we will set  the initial condition at an arbitrary  time $t_0\in \mathbb R$, namely
we shall consider the Cauchy problem 
\begin{equation}\label{NLSt0}
\begin{cases}
i\partial_t u + \partial_x^2 u - \, u |u|^{2p}=0, \quad (t,x)\in \mathbb{R} \times \mathbb{T},\\
u(t_0,x)=u_0(x)\in H^k(\mathbb T).
\end{cases}
\end{equation}
Along the rest of the paper we use the following norm associated with $H^k(\mathbb T)$
$$
\|u\|^2_{H^k(\mathbb T)}=\|u\|_{L^2(\mathbb T)}^2+ \|\partial_x^k u\|_{L^2(\mathbb T)}^2.
$$
Our main result is the following statement.
\begin{theorem}\label{main}
Let $t_0\in\R$, $k,p\in \mathbb N$ with $k\geq 1$ and $p\geq 2$, $R>0$ and $\varepsilon>0$ be fixed. Then there exist $C,T>0$ and an energy
$${\mathcal E}_k: H^k(\mathbb{T})\longrightarrow \mathbb R,$$
with the structure
\begin{equation}\label{1}
{\mathcal E}_k(u)=\|u\|_{H^k(\mathbb T)}^2 + {\mathcal F}_k(u)
\end{equation}
such that the following property holds:
if  $u(t,x)\in {\mathcal C}(\R; H^k(\T))$  is the unique solution to \eqref{NLSt0} with 
$$
\|u(t_0,\cdot)\|_{H^1(\T)}\leq R
$$  
then we have:
\begin{equation}\label{2}\sup_{t\in (t_0-T, t_0+T)}|{\mathcal F}_k(u(t))|\leq C \|u(t_0)\|_{H^k(\mathbb T)}^{\frac{2k-4}{k-1}}\,,
\end{equation}
\begin{equation}\label{3}{\mathcal E}_k(u(t))\leq {\mathcal E}_k(u(t_0)) + C \|u(t_0)\|^{\frac{2k-4}{k-1}+\varepsilon}_{H^k(\mathbb T)}, \quad \forall t\in (t_0-T, t_0+T).\end{equation}
\end{theorem}
The proof of Theorem \ref{main} is an elaboration on the techniques introduced in \cite{PTV,PTV2}.
In the case $k\neq 3m$ ($m\in \Z$), we use the rather classical Zygmund $L^4$ space-time bound for linear solutions to the Schr\"odinger equation combined with the analysis of the invertibility of a matrix with dimension growing along with $k$. 
Compared with \cite{PTV2}, we have to deal with a complex valued $u$: for such $u$ the construction of the modified energies is significantly more involved. In the case 
$k=3m$ 
we have to use slightly more delicate dispersive tools involving Bourgain's spaces, to deal with unfavorable terms that cannot be eliminated by integrations by parts. Along the proof we shall also see that, except for the case $k= 3m$, we can set $\varepsilon=0$ in Theorem \ref{main}.
\\

As a byproduct of Theorem \ref{main} we get the following consequence about the polynomial upper bound on the growth of Sobolev norms of the solutions of \eqref{e:T1-NLS}.
\begin{theorem}\label{corcor}
Let $k,p\in \mathbb N$ with $k\geq 1$ and $p\geq 2$, $R>0$ and $\varepsilon>0$ be fixed.
Then there exists $C>0$ 
such that  for every $u_0\in H^k(\mathbb{T})$ such that
$\|u_0\|_{H^1(\mathbb{T})}\leq R$ we have the bound
\begin{equation}\label{e:sharp-polyn-growth}
\|u(t)\|_{H^k(\mathbb T)}\leq C \|u_0\|_{H^k(\mathbb T)} + Ct^{\frac{k-1}2+\varepsilon}
\end{equation}
where $u(t,x)$ is the unique global solution to \eqref{e:T1-NLS}.
Moreover, in the case $k\neq 3m$ ($m\in \Z$), we have \eqref{e:sharp-polyn-growth}
 with $\varepsilon=0$.
\end{theorem}
The proof of the Theorem \ref{corcor} is classical once Theorem \ref{main} is established (see e.g. \cite{PTV_first}). 
In the case $p=2$,  Corollary \ref{corcor} recovers the result of \cite{B2} with a different proof. In the case $p\geq 3$ the result of 
 Corollary \ref{corcor} is new and improves on \cite{CKO} where the following weaker bound is obtained:
\begin{equation}\label{ANPDE}
\|u(t)\|_{H^k(\mathbb T)}\leq C \|u_0\|_{H^k(\mathbb T)} + Ct^{k-1+\varepsilon}\, .
\end{equation}
Observe that the bound \eqref{ANPDE} is a consequence of \cite{PTV_first} where the same bound is established for solutions of NLS on $\T^2$ by simply considering a data independent of one of the variables.
\\

{\bf Notations.} Along the paper we shall denote by $H^k$ the Sobolev space $H^k(\mathbb T)$ and more generally by $W^{k,p}$  the Sobolev space $W^{k,p}(\mathbb T)$. We denote by $e^{it\partial_x^2}$ the group associated with the linear Schr\"odinger equation with periodic boundary condition.
For any $p\in [1, \infty]$, $(X,\|\cdot\|)$ Banach space and  $I\subset \R$ interval, we denote by $\|u\|_{L^p(I; X)}$, where $u$ is any space-time function, the norm defined by $(\int_I \|u(t)\|_{X}^p dt)^\frac 1p$ for $p\neq \infty$, with the usual modification for $p=\infty$. For shortness we shall write $\int f=\int_\T f(x) dx$ for any time--independent periodic function and for time-dependent functions $\int_I\int f =\int_I \int_\T f dtdx$ for any time interval $I\subset \R$.
\section{Preliminary facts on the Cauchy theory}\label{prel}
In this section we gather together some well-known facts about solutions to \eqref{NLSt0}
that will be useful in the sequel.
We consider two different functional settings, the first one is the usual Sobolev spaces, the second one is the more sophisticated Bourgain's $X^{s,b}$ spaces framework. This last framework is needed to establish Theorem \ref{main} and \ref{corcor}
in the case $k=3m$, all the other cases do not require the use of the $X^{s,b}$ spaces.
\subsection{The classical Sobolev setting}
\begin{proposition}\label{cauchyprop}
For every $k\geq 1$, $R>0$ there exist $C,T>0$ such that for every $t_0\in \mathbb R$
and for every $u_0\in H^k$ with $\|u_0\|_{H^1}\leq R$ 
the Cauchy problem \eqref{NLSt0} admits one unique local solution
\begin{equation}\label{sobol}
u(t,x)\in {\mathcal C}((t_0-T, t_0+T); H^k) \hbox{ s.t. } \sup_{t\in (t_0-T, t_0+T)} 
\|u(t)\|_{H^k}\leq C \|u(t_0)\|_{H^k}.\end{equation}
Moreover the unique solution $u(t,x)$ satisfies the bound
\begin{equation}\label{zygmund}
\|u(t,x)\|_{L^4((t_0-T, t_0+T); W^{k,4})} \leq C \|u(t_0)\|_{H^k}.
\end{equation} 
Finally the local solution $u(t,x)$ can be extended globally in time, in particular
$u(t,x)\in {\mathcal C}(\R; H^k)$.
\end{proposition}
\begin{proof}[Sketch of Proof]
Recall that the existence of one unique solution local in time for \eqref{NLSt0}, for initial datum in $H^1$, follows by combining
the Sobolev embedding $H^1\subset L^\infty$ along with the fact that $e^{it\partial_x^2}$ is an isometry in $H^1$. As a consequence
one can show that the integral nonlinear operator
$$u\longmapsto e^{it\partial_x^2} u(t_0)- i \int_{t_0}^{t_0+t} e^{i(t-s+t_0)\partial_x^2} (u(s)|u(s)|^{2p}) ds$$
is a contraction in a suitable ball of the space $\mathcal C((t_0-T, t_0+T); H^1)$, provided $T=T(R)$ is sufficiently small. 
In particular the Duhamel operator associated with \eqref{NLSt0} admits a unique fixed point in this ball.
This concludes the proof of the existence 
of a local solution in $H^1$.
The uniqueness in  $\mathcal C((t_0-T, t_0+T); H^1)$ follows by similar arguments. 
Based on the conservation of energy (which is non-negative since we are working with the defocusing NLS)  we also have  
\begin{equation}\label{energyest}\sup_t \|u(t)\|_{H^1}<\infty\end{equation}
hence by an iteration of the local existence result stated above the solution $u$ is global in time. 
By using tame estimates one can show that if the initial datum belongs to $H^k$ with $k>1$ then the regularity is propagated
 and in particular \eqref{sobol} holds.
The proof of the  bound \eqref{zygmund} follows by combining \eqref{sobol} with 
the integral formulation associated with \eqref{NLSt0} along with the classical  linear
bound by Zygmund (see \cite{Zy}):
$$
\|e^{it \partial_{x}^2} \varphi\|_{L^4((0,1); L^4)}\leq C \|\varphi\|_{L^2}.
$$
\end{proof}
\subsection{The $X^{s,b}$ setting}
We first recall the definition of Bourgain's spaces $X^{s,b}_T$ introduced in \cite{Bo_gafa}. 
First we introduce the (global in time) $X^{s,b}$ space equipped with the norm
$$
\| u\|_{X^{s,b}}^2=\sum_{n\in\Z}\int_{\R}\langle \tau+n^2\rangle^{2b}\langle n\rangle^{2s}|\widehat{u}(\tau,n)|^2\, d\tau\,.
$$
For every finite interval $I\subset \R$ we introduce the time localized version $X^{s,b}_I$  of the spaces $X^{s,b}$, defined as the space of distributions on $I\times \T$  that can be extended 
to a global function belonging to $X^{s,b}$. We equip this space with the norm
$$
\| u\|_{X^{s,b}_I}=\inf \| \tilde{u}\|_{X^{s,b}},
$$
where the $\inf$ is taken over all global extensions $\tilde{u}\in X^{s,b}$ of $u$. 
For $b>1/2$ the space  $X^{s,b}_I$ is embedded in ${\mathcal C}(I;H^s)$. 
The Cauchy theory for \eqref{NLSt0} can be strengthened within the $X^{s,b}$ spaces as follows.
\begin{proposition}\label{cauchypropXsb}
For every $s\geq 1$, $R>0$  and $b>\frac 12$  there exist $C,T>0$ such that for every $t_0\in \mathbb R$
and for every $u_0\in H^s$ with $\|u_0\|_{H^1}\leq R$ 
the Cauchy problem \eqref{NLSt0} admits one unique local solution
\begin{equation}\label{sobolXsb}
u(t,x)\in X^{s,b}_{(t_0-T, t_0+T)} \hbox{ s.t. } \|u\|_{X^{s,b}_{(t_0-T, t_0+T)}} \leq C \|u(t_0)\|_{H^s}\,.
\end{equation}
Moreover the solution $u(t,x)$ can be extended globally in time and belongs to the space
$X^{k,b}_{I}$ for every bounded time interval $I\subset \R$.
\end{proposition} 

\section{Construction of the energies ${\mathcal E}_k$}

We start this section by introducing suitable densities that will be useful either to define the energy $\mathcal E_k(u(t,\cdot))$
or to compute $\frac d{dt} \mathcal E_k(u(t,\cdot))$ along the flow of \eqref{NLSt0}.\\
\\
\begin{definition}\label{defgammak} For a fixed $(i,j)=(i_1, \ldots, i_n;j_1, \ldots, j_n)\in \mathbb N^n \times \mathbb N^n$, let $\mathcal J_{(i,j)}(u)$ be the functional
\[
\mathcal J_{(i,j)}(u)=\int  \partial_x^{i_1} u\, \cdots \partial_x^{i_n}u \, \partial_x^{j_1}\bar u\cdots \partial_x^{j_n}\bar u.
\]
Then we set 
\begin{multline*}
\Theta_k:=\left\{\mbox{ linear combinations of } \, \mathcal J_{(i,j)}(u):\, (i,j)\in \mathcal{C}^k\right\},
\\
\Omega_k:=\left\{\mbox{ linear combinations of } \, \mathcal J_{(i,j)}(u):\, (i,j)\in \mathcal{D}^k\right\},
\\
\Gamma_k:=\left\{\mbox{ linear combinations of } \, \mathcal J_{(i,j)}(u):\, (i,j)\in \mathcal{G}^k\right\},
\end{multline*}
where $\mathcal C^k$, $\mathcal D^k$, $\mathcal G^k$ are the next sets of multi-indexes:
\begin{multline*}
\mathcal C^k := \Big\{(i,j) \in \mathbb N^{2p+1} \times \mathbb N^{2p+1}:\,\, i_1 \ge i_2 \ge \ldots \ge i_{2p+1},\\ j_1 \ge j_2 \ge \ldots \ge j_{2p+1},
\sum_{\ell =1}^{2p+1} i_\ell+ j_\ell =2k-2,\, i_1, \, j_1\le k-1\Big\},
\\
\mathcal{D}^k:=\Big\{(i,j)\in \mathbb{N}^{p+1}\times \mathbb{N}^{p+1}:\, i_1 \ge i_2 \ge \ldots \ge i_{p+1}, \\ j_1 \ge j_2 \ge \ldots \ge j_{p+1},
\sum_{\ell=1}^{p+1} i_\ell +j_\ell=2k,\, \sum_{\ell=1}^{p+1}\min\{ i_\ell, 1\} + \sum_{\ell=1}^{p+1}\min\{j_\ell,1\} \ge 4 \Big\}\\
\mathcal G^k := \Big\{(i,j) \in \mathbb N^{p+1} \times \mathbb N^{p+1}:\,\, i_1 \ge i_2 \ge \ldots \ge i_{p+1},\\ j_1 \ge j_2 \ge \ldots \ge j_{p+1},
\sum_{\ell =1}^{p+1} i_\ell+ j_\ell =2k-2,\, i_1, \, j_1\le k-1\Big\}.
\end{multline*}
\end{definition}
\begin{remark}We emphasize that the last condition in the definition of $\mathcal D^k$ makes densities in $\Omega_k$ having in itself products of (at least) four nontrivial derivatives, while the last condition in the definition of $\mathcal C^k$ and $\mathcal G^k$ means that densities in $\Theta_k$ admit $k-1$ as the highest possible number of derivatives. 
\end{remark}
Next we introduce special densities 
 that appear 
from the time-derivative of $\|\partial_x^k u(t)\|^2_{L^{2}}$ (see Lemma \ref{Hk}). 
\begin{definition}
Let $k=3m+r$, for some $m \in \mathbb{N}$ and $r \in\{0,1,2\}$. For $h=0,\dots, m$, we define the following densities:
\begin{equation}\label{e:basic-densities}
  \left\{
    \begin{aligned}
\mathcal I_{k,h}(u):= & \mathrm{Im}\int \big(\partial_x^{k-h} u\big)^2\, \partial_x^{2h} u\, u^{p-2}\, \bar u^{p+1},\\
\mathcal K_{k,h}(u):= &\mathrm{Im}\int \big(\partial_x^{k-h} u\big)^2\, \partial_x^{2h} \bar u\, u^{p-1}\, \bar u^{p},
\\
\mathcal V_{k,h}(u):= &\mathrm{Im}\int \partial_x^{k-h} u\, \partial_x^{k-h-1}\bar u\,\partial_x^{2h+1} u\, u^{p-1}\, \bar u^{p},
\\
\mathcal W_{k,h}(u):= &\mathrm{Im}\int \partial_x^{k-h} u\, \partial_x^{k-h-1}\bar u\,\partial_x^{2h+1}\bar u\, u^{p}\, \bar u^{p-1}.
\end{aligned}
\right.
\end{equation}
\end{definition}

We define now some specific densities belonging to $\Gamma_k$ (see definition \ref{defgammak}) 
that we shall use to complete the quantity
$\|u\|_{L^2}^2+\|\partial_x^k u\|_{L^2}^2$ to the energy $\mathcal E_k$.
\begin{definition}
  Let $k=3m +r$, with $m\in \mathbb{N}$ and $r\in \{0,1,2\}$. For $h=1,\dots, m+1$, we define
  \begin{equation}
    \label{e:corrections}\left\{
\begin{aligned}
\tilde{\mathcal I}_{k,h}(u):= &\mathrm{Re}\int \big(\partial_x^{k-h} u\big)^2\, \partial_x^{2h-2}u\, \bar u^{p+1}\, u^{p-2},\\
\tilde{\mathcal K}_{k,h}(u):=& \mathrm{Re}\int \big(\partial_x^{k-h} u\big)^2\, \partial_x^{2h-2}\bar u\, \bar u^{p}\, u^{p-1},
\\
\tilde{\mathcal V}_{k,h}(u):= &\mathrm{Re} \int |\partial_x^{k-h}u|^2\, \partial_x^{2h-2}u\, \bar u |u|^{2(p-1)},\\
\tilde{\mathcal W}_{k,h}(u):= &\mathrm{Re}\int \partial_x^{k-h} u\,\partial_x^{k-h-1} \bar u\, \partial_x^{2h-1}u\, \bar u^{p}\, u^{p-1}.
\end{aligned}\right.
\end{equation}
\end{definition}

\begin{remark}\label{remequalterms}
It is worth noting that the densities in \eqref{e:basic-densities} and in \eqref{e:corrections} may be redundant.
For instance $\mathcal I_{k,0}(u)=\mathcal K_{k,0}(u)$ and $\tilde{\mathcal I}_{2,1}=\tilde{\mathcal K}_{2,1}=\tilde{\mathcal W}_{2,1}$, or some of them can be written as linear combination of others. However, we prefer to \emph{not} detail all these occurrences, specifying the relations between the densities at the moment when we take advantage of them.
\end{remark}
We now introduce the modified energy $\mathcal E_k(u)$ of Theorem \ref{main}
and we describe its main properties.
\begin{proposition}
\label{prop:structure}
Let $k=3m+r$, with $m \in \mathbb{N}$ and $r \in\{0,1,2\}$, with $k\ge 2$. Let $p \ge 2$ be an integer and assume that $u$ is a solution of \eqref{NLSt0}.
Then, there exist $\tilde \alpha_h, \tilde \beta_h, \tilde \gamma_h\in \mathbb R$, with $h =0,\ldots , m$ and
$\tilde \delta_h\in \mathbb R$ with $h =1,\ldots , m$ such that, after we set
\begin{multline}\label{defEk}
\mathcal{E}_k(u):= \|u\|_{H^k}^2 + \sum_{h=0}^{m} \left(\tilde{\alpha}_h\, \tilde{\mathcal I}_{k,h+1}(u) +\tilde{\beta}_h\, \tilde{\mathcal V}_{k,h+1}(u)+\tilde{\gamma}_h\, \tilde{\mathcal W}_{k,h+1}(u)\right)
\\+\sum_{h=1}^{m} \tilde{\delta}_h\, \tilde{\mathcal K}_{k,h+1}(u),
\end{multline}
the following holds:
\begin{itemize}
\item when $k=3m+1$ or $k=3m+2$:
\begin{equation}
\label{e:sharp-estimate}
\frac{d}{dt}\mathcal E_k(u) = \sum_{(i,j)\in \mathcal D^k}\lambda_{(i,j)}\, \mathcal J_{(i,j)}(u) + \sum_{(i,j)\in \mathcal C^k}\mu_{(i,j)} \, \mathcal J_{(i,j)}(u),
\end{equation}
for some real numbers $\lambda_{(i,j)}$ and  $\mu_{(i,j)}$;
\item when $k=3m$:
\begin{multline}
\label{e:estimate-with-remains}
\frac{d}{dt}\mathcal E_k(u) = c\, \mathrm{Im}\int \big(\partial_x^{2m} u\big)^3\, u^{p-2}\, \bar u^{p+1} \\
+ \sum_{(i,j)\in \mathcal D^k}{\tilde\lambda}_{(i,j)}\, \mathcal J_{(i,j)}(u) + \sum_{(i,j)\in \mathcal C^k}{\tilde\mu}_{(i,j)} \, \mathcal J_{(i,j)}(u),
\end{multline}
for some real numbers $c$, ${\tilde \lambda}_{(i,j)}$ and  ${\tilde \mu}_{(i,j)}$.
\end{itemize}
\end{proposition}
\begin{remark}
Using the notations introduced above, the relevant features of 
the energy ${\mathcal E}_k(u)$ can be summarized as follows:
$${\mathcal E}_k(u)-\|u\|_{H^k}^2 \in \Gamma_k$$
and
$$
\frac{d}{dt}\mathcal E_k(u) \in \Omega_k + \Theta_k, \quad k=3m+1 \hbox{ or } 3m+2$$
$$
\frac{d}{dt}\mathcal E_k(u)- c\,  \mathrm{Im}\int \big(\partial_x^{2m} u\big)^3\, u^{p-2}\, \bar u^{p+1} \in \Omega_k + \Theta_k, \quad k=3m.$$
\end{remark}
\noindent Next we define our last useful densities.
\begin{definition} Given a density
\begin{equation}\label{Ju}
\mathcal J(u)=\int  \partial_x^{i_1} u\, \cdots \partial_x^{i_n}u \, \partial_x^{j_1}\bar u\cdots \partial_x^{j_n}\bar u,\end{equation}
we define $\mathcal J^*(u)$ and $\mathcal J^{**}(u)$ as the densities resulting by considering the time derivative of the density $\mathcal J$ and by replacing $\partial_t u$ with $i\,\partial_x^2 u$ and $\partial_t \bar u$ with $-i\, \partial_x^2 \bar u$. Similarly $\mathcal J^{**}$ is obtained as $\mathcal J^*$, but by replacing $\partial_t u$ with $-i\,|u|^{2p}\,u$ and $\partial_t \bar u$ with $i\,|u|^{2p}\, \bar u$. 
\end{definition}
\noindent These definitions make the following identity true, if  $u$ satisfies \eqref{NLSt0}:
\begin{equation}
\label{e:asts}
\frac{d}{dt} \mathcal{J}(u)=\mathcal{J}^{\ast}(u) +
\mathcal{J}^{\ast\ast}(u).
\end{equation}
We conclude by the following useful notation that will be used along the rest of the paper. 
\begin{notation}
For any couple of densities
$\mathcal D_1(u)$ and $\mathcal D_2(u)$ which are respectively linear combination of terms of the type \eqref{Ju} we will use the following equivalence
$${\mathcal D}_1(u)\equiv {\mathcal D}_2(u) \iff {\mathcal D}_1(u)- {\mathcal D}_2(u)
\in \Omega_k+\Theta_k.$$
\end{notation}
\begin{remark}
Following this notation we have that by \eqref{e:sharp-estimate}
$$\frac d{dt} \mathcal E_k(u)\equiv 0 \hbox{ for } k=3m+1, 3m+2$$
and by 
\eqref{e:estimate-with-remains}
$$\frac d{dt} \mathcal E_k(u)\equiv c\, \mathrm{Im}\int \big(\partial_x^{2m} u\big)^3\, u^{p-2}\, \bar u^{p+1} \hbox{ for } k=3m.$$
\end{remark}
%
\subsection{Proof of Proposition \ref{prop:structure}}

We need several lemmas.

\begin{lemma}
\label{Hk}
Let $k=3m+r$, with $m \in \mathbb{N}$, $r \in\{0,1,2\}$, $k\ge 2$. Assume that $u(t,x)$ is a smooth solution of \eqref{NLSt0}. 
Then, there exist $\alpha_h, \beta_h, \gamma_h\in \mathbb{R}$ for $h=0,\ldots, m$ and 
$\delta_h \in \mathbb{R}$
for $h=1,\ldots, m$, $\lambda_{(i,j)}\in \mathbb{R}$ for $(i,j)\in \mathcal{D}^k$ such that:
\begin{multline}
\label{e:der}
\frac{d}{dt} \|\partial_x^k u\|_{L^2}^2 =
 \sum_{h=0}^{m}\Big(\alpha_h\, \mathcal I_{k,h}(u)+ \beta_h\, \mathcal V_{k,h}(u)+ \gamma_h\, \mathcal W_{k,h}(u) \Big)\\
 +\sum_{h=1}^m \delta_h\, \mathcal K_{k,h}(u)
+\sum_{(i,j)\in\mathcal{D}^k}\lambda_{(i,j)}\, \mathcal J_{(i,j)}(u).
\end{multline}
\end{lemma}

\begin{proof} We start by noticing that in the r.h.s. in \eqref{e:der} we have the sum on $\mathcal K_{k,h}(u)$ running from $h=1$ to $h=m$,
while on the other terms the sum includes $h=0$. This is due to the fact that, as noticed in remark \ref{remequalterms}, we have
that the term $\mathcal K_{k,0}(u)$ is equal to $\mathcal I_{k,0}(u)$ hence the contribution given by $\mathcal K_{k,0}(u)$
can be absorbed by $\mathcal I_{k,0}(u)$. Next we prove \eqref{e:der}.
By using the equation solved by $u(t,x)$ we get:
\begin{multline*}
\frac{d}{dt} \|\partial_x^k u\|_{L^2}^2 = 2\, \mathrm{Re} \int \partial_x^k u \, \partial_x^k \partial_t \bar u
\\=- 2 \, \mathrm{Re} \, i \int \partial_x^k u \, \partial_x^{k+2} \bar u + 2 \, \mathrm{Re} \, i \int \partial_x^k u \, \partial_x^k \big(u^p \, \bar u^{p+1}\big) 
=-2\, \mathrm{Im}\int \partial_x^k u \, \partial_x^k \big(u^p \, \bar u^{p+1}\big).
\end{multline*}
If we develop $\partial_x^k \big(u^p \, \bar u^{p+1}\big)$
then we write $\mathrm{Im}\int \partial_x^k u \, \partial_x^k \big(u^p \, \bar u^{p+1}\big)$ as linear combination of several terms, some of them belonging to $\Omega_k$ (and hence absorbed in the last term of 
\eqref{e:der}) while the others are linear combinations of terms of the type 
\begin{equation}\label{termsabo}\mathrm{Im} \, \int (\partial_x^k u \partial_x^\beta v_1 \partial^\gamma_x v_2) \times \prod_{j=3}^{2p+1} v_j
\hbox{ where } v_j\in\{u, \, \bar u\}, \quad k+\beta+\gamma=2k\end{equation}
and moreover the number of factors involving $u$ is equal to the number of factors involving $\bar u$. 
By using integration by parts we have that the terms in \eqref{termsabo}, up to terms that belong to $\Omega_k$ (so absorbed in the last term in \eqref{e:der}), are linear combinations of terms of the type
\begin{equation}\label{...}\mathrm{Im} \, \int (\partial_x^{\alpha_1+1} u \partial_x^{\alpha_1} v_1 \partial^{\gamma_1}_x v_2 ) \times \prod_{j=3}^{2p+1} v_j, \quad 
\alpha_1\geq \gamma_1\geq 0, \quad 2\alpha_1+\gamma_1+1=2k\end{equation}
or of the following type
\begin{equation}\label{......}\mathrm{Im} \, \int (\partial_x^{\alpha_2} u \partial_x^{\alpha_2} v_1  \partial^{\gamma_2}_x v_2 )\times \prod_{j=3}^{2p+1} v_j, \quad
\alpha_2\geq \gamma_2\geq 0, \quad 2\alpha_2+\gamma_2=2k.\end{equation}
We split now in several cases.
\\
{\bf First case: $\alpha_1> \gamma_1$ in \eqref{...}, $\alpha_2> \gamma_2$ in \eqref{......}.} 
\\
Notice that concerning the terms in \eqref{...} we have two possibilities: either
$v_1=u$ or $v_1=\bar u$. In the first case we can argue  
by integration by parts and \eqref{...} can be written, up to terms belonging to $\Omega_k$ and up to a multiplicative factor, as 
$\Im \int ((\partial_x^{\alpha_1} u)^2 \partial_x^{\gamma_1+1} v_2) \times \prod_{j=3}^{2p+1} v_j$
which is a term of the type $\mathcal I_{k,h}(u)$ or $\mathcal K_{k,h}(u)$ depending on $v_2$, with $h=\frac{\gamma_1+1}2$. 
Notice that by the condition $2\alpha_1+\gamma_1+1=2k$ 
we get that $h$ is an integer. Moreover necessarily $h\leq m$, in fact by the conditions in \eqref{...} we get $2\alpha_1+\gamma_1+1=6m + 2r$ and $\alpha_1>\gamma_1$ 
which imply $2h-1 + h < \alpha_1 + h=3m + r$. In the second case $v_1=\bar u$ we get that \eqref{...} reduces to a term of the type
$\mathcal W_{k,h}(u)$ or $\mathcal V_{k,h}(u)$ depending on $v_2$, with $2h+1=\gamma_1$. Arguing as above also in this case necessarily $h\leq m$.
Next we focus on \eqref{......} by considering again two cases: either $v_1=u$ or $v_1=\bar u$. In the case $v_1=u$ we have \eqref{......}  belongs to $\mathcal I_{k,h}(u)$ or $\mathcal K_{k,h}(u)$ depending on $v_2$, with $2h=\gamma_2$
and also in this case one can check $h\leq m$.
In the case $v_1=\bar u$ we have two more possibilities: either $\gamma_2=0$ and in this case the term \eqref{......} is equal to zero since we consider the imaginary part of a real number (remember that the amount of factors $u$ is equal to
the amount of factors $\bar u$) or we get $\gamma_2>0$ and we can integrate by parts, then up to extra terms in $\Omega_k$
we get that \eqref{......} becomes
$$-\mathrm{Im} \, \int (\partial_x^{\alpha_2} u \partial_x^{\alpha_2+1} \bar u  \partial^{\gamma_2-1}_x v_2) \times \prod_{j=3}^{2p+1} v_j
-\mathrm{Im} \, \int (\partial_x^{\alpha_2+1} u \partial_x^{\alpha_2} \bar u  \partial^{\gamma_2-1}_x v_2) \times \prod_{j=3}^{2p+1} v_j$$
which is a linear combination of terms belonging to 
$\mathcal W_{k,h}(u)$ or $\mathcal V_{k,h}(u)$ depending on $v_2$ with $2h+1=\gamma_2-1$ and also in this case we get $h\leq m$.
\\
{\bf Second case: $\alpha_1=\gamma_1$ in \eqref{...}, $\alpha_2=\gamma_2$ in \eqref{......}.} 
\\
The expression \eqref{...} reduces to 
$$\mathrm{Im} \, \int( \partial_x^{\alpha_1+1} u \partial_x^{\alpha_1} v_1 \partial^{\alpha_1}_x v_2) \times \prod_{j=3}^{2p+1} v_j.$$
It is easy to see that if $v_1=v_2=u$ then by integration by parts we get a term belonging to $\Omega_k$.
In case $v_1=v_2=\bar u$ then we get  $\mathcal W_{k,h}(u)$ with $2h+1=\alpha_1$. It is easy to check, by using the condition on $\alpha_1$, that in this case $h$ is an integer and $h\leq m$.
Finally in the case $v_1$ and $v_2$ different then we get $\mathcal W_{k,h}(u)$ with $2h+1=\alpha_1$, and again $h$  in an integer with $h\leq m$. 
Next we focus on the term \eqref{......} that reduces to
$$\mathrm{Im} \, \int (\partial_x^{\alpha_2} u \partial_x^{\alpha_2} v_1  \partial^{\alpha_2}_x v_2) \times \prod_{j=3}^{2p+1} v_j$$
and in this case we get either a term $\mathcal I_{k,h}(u)$ or a term $\mathcal K_{k,h}(u)$ with $2h=\alpha_2$. It is easy to check that necessarily $h$ is an integer and $h\leq m$.
\end{proof}

\begin{lemma}
\label{lem:1}
For any given $k>1$ we have
\begin{equation} 
\label{e:ast-to-worse1}
{\tilde{\mathcal I}_{k,1}^\ast}(u)\equiv 2\, \mathcal I_{k,0}(u)-2(p-1)\, \mathcal I_{k,1}(u),\quad 
{\tilde{\mathcal V}_{k,1}^\ast}(u)\equiv  4p\, \mathcal V_{k,0}(u).
\end{equation}
Moreover for $k>2$ we have
\begin{equation}\label{e:ast-to-worse12}
{\tilde{\mathcal W}_{k,1}^\ast}(u)\equiv 2\,\mathcal V_{k,0}(u) - 2\,\mathcal W_{k,0}(u) -2\, \mathcal V_{k,1}(u).
\end{equation}
\end{lemma}

\begin{proof}
We begin with the equivalence involving ${\tilde{\mathcal I}}_{k,1}^\ast$. We have
\begin{multline*}
{\tilde{\mathcal I}_{k,1}^\ast}(u)
=
\left\{
\mathrm{Re}\int \big(\partial_x^{k-1} u\big)^2\, \bar u^{p+1}\, u^{p-1} \right\}^\ast
=-2\,\mathrm{Im}\int \partial_x^{k+1}u\, \partial_x^{k-1}u\, \bar u^{p+1}\, u^{p-1} 
\hbox{ where } 
\\
- \mathrm{Im}\int \big(\partial_x^{k-1} u\big)^2 \big(-(p+1)\, \partial_x^2 \bar u \,u +(p-1)\, \partial_x^2 u \, \bar u\big) \, u^{p-2}\, \bar u^p
\end{multline*}
and hence after integration by parts
\begin{multline*}
{\tilde{\mathcal I}_{k,1}^\ast}(u) =2\, \mathrm{Im}\int \big(\partial_x^{k} u\big)^2\, \bar u^{p+1}\, u^{p-1} 
\\+ 2\, \mathrm{Im}\int \partial_x^{k} u\,\partial_x^{k-1} u\, \big((p+1)\,\partial_x \bar u \, u+(p-1) \, \partial_x u \, \bar u\big)\, u^{p-2}\, \bar u^p
\\
-  \mathrm{Im}\int \big(\partial_x^{k-1} u\big)^2 \big(-(p+1)\, \partial_x^2 \bar u \,u +(p-1)\, \partial_x^2 u \, \bar u\big) \,u^{p-2} \bar u^p.
\end{multline*}
By integrating by parts the second integral we get 
\begin{multline*}
{\tilde{\mathcal I}_{k,1}^\ast}(u) \equiv 2\, \mathrm{Im}\int \big(\partial_x^{k} u\big)^2\, \bar u^{p+1}\, u^{p-1} \\ -\, \mathrm{Im}\int (\partial_x^{k-1} u)^2 \big((p+1)\,\partial_x^2 \bar u \, u+(p-1) \, \partial_x^2 u \, \bar u\big)\, u^{p-2}\, \bar u^p
\\
-  \mathrm{Im}\int \big(\partial_x^{k-1} u\big)^2 \big(-(p+1)\, \partial_x^2 \bar u \,u +(p-1)\, \partial_x^2 u \, \bar u\big) \,u^{p-2} \bar u^p
\end{multline*}
which in turn implies the first equivalence in \eqref{e:ast-to-worse1}.
\smallskip
Analogously we have
\begin{multline*}
  {\tilde{\mathcal V}_{k,1}^\ast}(u)
  =\left\{
\int |\partial_x^{k-1} u|^2 \,|u|^{2p}\right\}^\ast=-2
\mathrm{Im} \int \partial_x^{k+1}u\, \partial_x^{k-1}\bar u\, |u|^{2p}
\\ -2p\,\mathrm{Im} \int |\partial_x^{k-1}u|^2\, |u|^{2(p-1)}\,\partial_x^2 u\, \bar u,
\end{multline*}
which  gives, after integrate by parts
\begin{multline*}
  {\tilde{\mathcal V}_{k,1}^\ast}(u) \equiv  4p\,\mathrm{Im}\int \partial_x^{k}u\, \partial_x^{k-1}\bar u\, \mathrm{Re} \, (\partial_x u\, \bar u) \, |u|^{2(p-1)}  \\+2p\,\mathrm{Im} \int \partial_x (|\partial_x^{k-1}u|^2)\, |u|^{2(p-1)}\,\partial_x u\, \bar u\\
=4p\int \mathrm{Im} \, (\partial_x^{k}u\, \partial_x^{k-1}\bar u)\, \mathrm{Re} \, (\partial_x u\, \bar u) \, |u|^{2(p-1)}  \\+4p \int \mathrm{Re} \; (\partial_x^{k}u\, \partial_x^{k-1} \bar u )\, |u|^{2(p-1)} \,\mathrm{Im} (\partial_x u\, \bar u)
\\
=4p\, \mathrm{Im} \int \partial_x^{k}u\, \partial_x^{k-1}\bar u \, \partial_x u\, \bar u \, |u|^{2(p-1)}  
\end{multline*}
where at the last step we used the identity 
$\mathrm{Im}\, z \, \mathrm{Re} \, w+ \mathrm{Re} \, z \, \mathrm{Im} \, w= \mathrm{Im} \, (z\, w)$
for any couple $z,\, w\in \mathbb C$. Hence we get the second equivalence in \eqref{e:ast-to-worse1}.
\smallskip 
Next we prove \eqref{e:ast-to-worse12} by assuming $k> 2$:
\begin{multline*}
{\tilde{\mathcal W}_{k,1}^\ast}(u)
=\left\{\mathrm{Re}\int \partial_x^{k-1} u\,\partial_x^{k-2} \bar u\, \partial_x u\, \bar u^{p}\, u^{p-1}\right\}^\ast\\\equiv -{\mathrm{Im}}\int\partial_x^{k+1}u\, \partial_x^{k-2}\bar u\, \partial_x u \, \bar u^{p}\, u^{p-1}
\\
\qquad - \mathrm{Im} \int \partial_x^k u\, \partial_x^{k-1}\bar u \, \partial_x \bar u 
\, u^{p}\, \bar u^{p-1} - \mathrm{Im}\int \partial_x^{k-1}u\, \partial_x^{k-2}\bar u\, \partial_x^3 u 
\, \bar u^{p}\, u^{p-1}
\\
\equiv {\mathrm{Im}}\int\partial_x^{k}u\, \partial_x^{k-1}\bar u\, \partial_x u \, \bar u^{p}\, u^{p-1}+
{\mathrm{Im}}\int\partial_x^{k}u\, \partial_x^{k-2}\bar u\, \partial_x^2 u \, \bar u^{p}\, u^{p-1}
\\
\qquad - \mathrm{Im} \int \partial_x^k u\, \partial_x^{k-1}\bar u \, \partial_x \bar u 
\, u^{p}\, \bar u^{p-1} - \mathrm{Im}\int \partial_x^{k-1}u\, \partial_x^{k-2}\bar u\, \partial_x^3 u 
\, \bar u^{p}\, u^{p-1}
\end{multline*}
and by integration by parts
\begin{multline*}
  {\tilde{\mathcal W}_{k,1}^\ast}(u)
\equiv {\mathrm{Im}}\int\partial_x^{k}u\, \partial_x^{k-1}\bar u\, \partial_x u \, \bar u^{p}\, u^{p-1}-
{\mathrm{Im}}\int\partial_x^{k-1}u\, \partial_x^{k-1}\bar u\, \partial_x^2 u \, \bar u^{p}\, u^{p-1}
\\-
{\mathrm{Im}}\int\partial_x^{k-1}u\, \partial_x^{k-2}\bar u\, \partial_x^3 u \, \bar u^{p}\, u^{p-1} - \mathrm{Im} \int \partial_x^k u\, \partial_x^{k-1}\bar u \, \partial_x \bar u 
\, u^{p}\, \bar u^{p-1}\\ - \mathrm{Im}\int \partial_x^{k-1}u\, \partial_x^{k-2}\bar u\, \partial_x^3 u 
\, \bar u^{p}\, u^{p-1}
\end{multline*}
and we continue
\begin{multline*}
  {\tilde{\mathcal W}_{k,1}^\ast}(u)
\equiv {\mathrm{Im}}\int\partial_x^{k}u\, \partial_x^{k-1}\bar u\, \partial_x u \, \bar u^{p}\, u^{p-1}-
{\mathrm{Im}}\int\partial_x^{k-1}u\, \partial_x^{k-1}\bar u\, \partial_x^2 u \, \bar u^{p}\, u^{p-1}
\\-
{\mathrm{Im}}\int\partial_x^{k-1}u\, \partial_x^{k-2}\bar u\, \partial_x^3 u \, \bar u^{p}\, u^{p-1} + \mathrm{Im} \int \partial_x^{k-1} u\, \partial_x^{k}\bar u \, \partial_x \bar u \, u^{p}\, \bar u^{p-1}
\\+ \mathrm{Im} \int \partial_x^{k-1} u\, \partial_x^{k-1}\bar u \, \partial_x^2 \bar u 
\, u^{p}\, \bar u^{p-1} - \mathrm{Im}\int \partial_x^{k-1}u\, \partial_x^{k-2}\bar u\, \partial_x^3 u 
\, \bar u^{p}\, u^{p-1}
\end{multline*}
and again
\begin{multline*}
{\tilde{\mathcal W}_{k,1}^\ast}(u)= {\mathrm{Im}}\int\partial_x^{k}u\, \partial_x^{k-1}\bar u\, \partial_x u \, \bar u^{p}\, u^{p-1}
+2 \mathrm{Im} \int \partial_x^{k-1} u\, \partial_x^{k-1}\bar u \, \partial_x^2 \bar u 
\, u^{p}\, \bar u^{p-1}
\\-
{\mathrm{Im}}\int\partial_x^{k-1}u\, \partial_x^{k-2}\bar u\, \partial_x^3 u \, \bar u^{p}\, u^{p-1} + \mathrm{Im} \int \partial_x^{k-1} u\, \partial_x^{k}\bar u \, \partial_x \bar u \, u^{p}\, \bar u^{p-1}
\\ - \mathrm{Im}\int \partial_x^{k-1}u\, \partial_x^{k-2}\bar u\, \partial_x^3 u 
\, \bar u^{p}\, u^{p-1}.
\end{multline*}
By using again integration by part on the second term
we continue
\begin{multline*}
{\tilde{\mathcal W}_{k,1}^\ast}(u) \equiv {\mathrm{Im}}\int\partial_x^{k}u\, \partial_x^{k-1}\bar u\, \partial_x u \, \bar u^{p}\, u^{p-1}
\\-2 \mathrm{Im} \int \partial_x^{k} u\, \partial_x^{k-1}\bar u \, \partial_x \bar u 
\, u^{p}\, \bar u^{p-1}-2 \mathrm{Im} \int \partial_x^{k-1} u\, \partial_x^{k}\bar u \, \partial_x \bar u 
\, u^{p}\, \bar u^{p-1}
\\-
{\mathrm{Im}}\int\partial_x^{k-1}u\, \partial_x^{k-2}\bar u\, \partial_x^3 u \, \bar u^{p}\, u^{p-1} + \mathrm{Im} \int \partial_x^{k-1} u\, \partial_x^{k}\bar u \, \partial_x \bar u \, u^{p}\, \bar u^{p-1}
\\ - \mathrm{Im}\int \partial_x^{k-1}u\, \partial_x^{k-2}\bar u\, \partial_x^3 u 
\, \bar u^{p}\, u^{p-1}
\end{multline*}
and we conclude the proof of \eqref{e:ast-to-worse12}.
\end{proof}


\begin{lemma}
\label{lem:2}
Let $k>2$ and $h=1,\dots,m-1$ then we have:
\begin{equation}
  \label{e:ast-to-worse2}
  \left\{
    \begin{aligned}
{\tilde{\mathcal I}_{k,h+1}^\ast}(u)\equiv & 2\, {{\mathcal I}_{k,h}}(u)-2\, {{\mathcal I}_{k,h+1}}(u),\\
{\tilde{\mathcal K}_{k,h+1}^\ast}(u)\equiv & 2\, \mathcal{K}_{k,h}(u), \\
{\tilde{\mathcal V}_{k,h+1}^\ast}(u)\equiv  & 2 \mathcal V_{k,h}, \\
      {\tilde{\mathcal W}_{k,h+1}^\ast}(u)\equiv & 2\,\mathcal{V}_{k,h}(u)-2\,\mathcal{W}_{k,h}(u)-2\, \mathcal{V}_{k,h+1}(u)\,.
    \end{aligned}
    \right.
\end{equation}
\end{lemma}

\begin{proof}
We begin with the identity for ${\tilde {\mathcal I}_{k,h+1}^\ast}(u)$, then we compute
\begin{multline}\label{eq:utildop}
  {\tilde{\mathcal I}_{k,h+1}^\ast}(u)
  =
\left\{ \mathrm{Re}\int \big(\partial_x^{k-h-1} u\big)^2\, \partial_x^{2h}u\, u^{p-2}\, \bar u^{p+1}\right\}^\ast\\\equiv -2\, \mathrm{Im} \int \partial_x^{k-h+1}u\, \partial_x^{k-h-1}u\, \partial_x^{2h} u
\, u^{p-2}\, \bar u^{p+1}
\\- \mathrm{Im}\, \int \big(\partial_x^{k-h-1}u \big)^2\, \partial_x^{2h+2}u \, u^{p-2}\, \bar u^{p+1}
\end{multline} 
and by integration by parts we get
\begin{multline*}
{\tilde{\mathcal I}_{k,h+1}^\ast}(u) \equiv 2\, \mathrm{Im}\int \big(\partial_x^{k-h} u\big)^2\, \partial_x^{2h}u 
\, u^{p-2}\, \bar u^{p+1}+ 2\, \mathrm{Im}\int \partial_x^{k-h}u\, \partial_x^{k-h-1}u\, \partial_x^{2h+1}u
\, u^{p-2}\, \bar u^{p+1}
\\
-\mathrm{Im}\int \big(\partial_x^{k-h-1}u\big)^2\, \partial_x^{2h+2}u 
\, u^{p-2}\, \bar u^{p+1}\\= 2\, \mathrm{Im}\int \big(\partial_x^{k-h} u\big)^2\, \partial_x^{2h}u 
\, u^{p-2}\, \bar u^{p+1}- 2\, \mathrm{Im}\int (\partial_x^{k-h-1}u)^2\, \partial_x^{2h+2}u
\, u^{p-2}\, \bar u^{p+1},
\end{multline*}
hence the first equivalence in $\eqref{e:ast-to-worse2}$ follows.
\smallskip
Similarly, we have
\begin{multline*}
  {\tilde{\mathcal K}_{k,h+1}^\ast}(u)
  =
\left\{\mathrm{Re}\int \big(\partial_x^{k-h-1} u\big)^2\, \partial_x^{2h}\bar u \, u^{p-1}\, \bar u^p\right\}^\ast\\\equiv 
- 2\, \mathrm{Im} \int \partial_x^{k-h+1}u\, \partial_x^{k-h-1}u\, \partial_x^{2h} \bar u
\, u^{p-1}\, \bar u^p
+ \mathrm{Im}\, \int \big(\partial_x^{k-h-1}u \big)^2\, \partial_x^{2h+2}\bar u \, u^{p-1}\, \bar u^p,
\end{multline*}
from which after integration by parts 
\begin{multline*}
{\tilde{\mathcal K}_{k,h+1}^\ast}(u) \equiv 2\, \mathrm{Im}\int \big(\partial_x^{k-h} u\big)^2\, \partial_x^{2h}\bar u 
\, u^{p-1}\, \bar u^p
+2\, \mathrm{Im}\int \partial_x^{k-h}u\, \partial_x^{k-h-1}u\, \partial_x^{2h+1}\bar u
\, u^{p-1}\, \bar u^p
\\
+\mathrm{Im}\int \big(\partial_x^{k-h-1}u\big)^2\, \partial_x^{2h+2}\bar u \, u^{p-1}\, \bar u^p\\
\equiv 2\, \mathrm{Im}\int \big(\partial_x^{k-h} u\big)^2\, \partial_x^{2h}\bar u 
\, u^{p-1}\, \bar u^p
-\mathrm{Im}\int \big(\partial_x^{k-h-1}u\big)^2\, \partial_x^{2h+2}\bar u \, u^{p-1}\, \bar u^p
\\+\mathrm{Im}\int \big(\partial_x^{k-h-1}u\big)^2\, \partial_x^{2h+2}\bar u \, u^{p-1}\, \bar u^p,
\end{multline*}
that is the second equivalence in \eqref{e:ast-to-worse2}.
\smallskip 
About the third equivalence in \eqref{e:ast-to-worse2}, we have
\begin{multline*}
  {\tilde{\mathcal V}_{k,h+1}^\ast}(u)
  =
\left\{ \mathrm{Re} \int |\partial_x^{k-h-1}u|^2\, \partial_x^{2h}u \, \bar u \, |u|^{2(p-1)}\right\}^\ast\\
\equiv -2\, \mathrm{Im} \int \partial_x^{k-h+1}u\, \partial_x^{k-h-1}\bar u\, \mathrm{Re} (\partial_x^{2h}u \, \bar u) \, |u|^{2(p-1)}
\\-\mathrm{Im} \int |\partial_x^{k-h-1}u|^2\,\partial_x^{2h+2}u \, \bar u \, |u|^{2(p-1)}\\
\equiv 
2\, \mathrm{Im} \int \partial_x^{k-h}u\, \partial_x^{k-h-1}\bar u\,\mathrm{Re} (\partial_x^{2h+1}u \, \bar u) \, |u|^{2(p-1)}\\
-\mathrm{Im} \int |\partial_x^{k-h-1}u|^2\,\partial_x^{2h+2}u \, \bar u \, |u|^{2(p-1)}
\end{multline*}
and we can continue
\begin{multline*}
{\tilde{\mathcal V}_{k,h+1}^\ast}(u) \equiv 
\mathrm{Im} \int \partial_x^{k-h}u\, \partial_x^{k-h-1}\bar u\,\partial_x^{2h+1}u \, \bar u \, |u|^{2(p-1)}
\\+ \mathrm{Im} \int \partial_x^{k-h}u\, \partial_x^{k-h-1}\bar u\,\partial_x^{2h+1}\bar u \,u \, |u|^{2(p-1)}
\\
+\mathrm{Im} \int \partial_x^{k-h}u \partial_x^{k-h-1}\bar u \,\partial_x^{2h+1}u \, \bar u \, |u|^{2(p-1)}
\\+\mathrm{Im} \int \partial_x^{k-h-1}u \partial_x^{k-h}\bar u \,\partial_x^{2h+1}u \, \bar u \, |u|^{2(p-1)}
\end{multline*}
and the third equivalence follows.
Finally we prove the fourth equivalence in \eqref{e:ast-to-worse2}. We have
\begin{multline*}
  {\tilde{\mathcal W}_{k,h+1}^\ast}(u)
  = \Big\{ \mathrm{Re}\int \partial_x^{k-h-1} u\,\partial_x^{k-h-2} \bar u\, \partial_x^{2h+1}u \, u^{p-1}\, \bar u^p\Big\}^\ast 
\\
\equiv - \mathrm{Im}\int\partial_x^{k-h+1}u\,\partial_x^{k-h-2} \bar u\, \partial_x^{2h+1}u 
\, u^{p-1}\, \bar u^p \\{}+ \mathrm{Im}\int\partial_x^{k-h-1}u\,\partial_x^{k-h} \bar u\, \partial_x^{2h+1}u
\, u^{p-1}\, \bar u^p\\{}- \mathrm{Im}\int \partial_x^{k-h-1} u\,\partial_x^{k-h-2} \bar u\, \partial_x^{2h+3}u 
\, u^{p-1}\, \bar u^p.
\end{multline*}
Hence by integration by parts we obtain
\begin{multline*}
{\tilde{\mathcal W}_{k,h+1}^\ast}(u) \equiv \mathrm{Im}\int\partial_x^{k-h}u\,\partial_x^{k-h-1} \bar u\, \partial_x^{2h+1}u 
\, u^{p-1}\, \bar u^p\\+\mathrm{Im}\int\partial_x^{k-h}u\,\partial_x^{k-h-2} \bar u\, \partial_x^{2h+2}u 
\, u^{p-1}\, \bar u^p\\
- \mathrm{Im}\int\partial_x^{k-h}u\,\partial_x^{k-h-1} \bar u\, \partial_x^{2h+1}\bar u 
\, \bar u^{p-1}\, u^p- \mathrm{Im}\,\int \partial_x^{k-h-1} u\,\partial_x^{k-h-2} \bar u\, \partial_x^{2h+3}u
\, u^{p-1}\, \bar u^p\\
\equiv \mathrm{Im}\int\partial_x^{k-h}u\,\partial_x^{k-h-1} \bar u\, \partial_x^{2h+1}u 
\, u^{p-1}\, \bar u^p\\
-\mathrm{Im}\int\partial_x^{k-h-1}u\,\partial_x^{k-h-1} \bar u\, \partial_x^{2h+2}u 
\, u^{p-1}\, \bar u^p
-\mathrm{Im}\int\partial_x^{k-h-1}u\,\partial_x^{k-h-2} \bar u\, \partial_x^{2h+3}u 
\, u^{p-1}\, \bar u^p\\
- \mathrm{Im}\int\partial_x^{k-h}u\,\partial_x^{k-h-1} \bar u\, \partial_x^{2h+1}\bar u 
\, \bar u^{p-1}\, u^p- \mathrm{Im}\,\int \partial_x^{k-h-1} u\,\partial_x^{k-h-2} \bar u\, \partial_x^{2h+3}u
\, u^{p-1}\, \bar u^p
\\
\equiv \mathrm{Im}\int\partial_x^{k-h}u\,\partial_x^{k-h-1} \bar u\, \partial_x^{2h+1}u 
\, u^{p-1}\, \bar u^p\\
+\mathrm{Im}\int\partial_x^{k-h}u\,\partial_x^{k-h-1} \bar u\, \partial_x^{2h+1}u 
\, u^{p-1}\, \bar u^p
+\mathrm{Im}\int\partial_x^{k-h-1}u\,\partial_x^{k-h} \bar u\, \partial_x^{2h+1}u 
\, u^{p-1}\, \bar u^p
\\-\mathrm{Im}\int\partial_x^{k-h-1}u\,\partial_x^{k-h-2} \bar u\, \partial_x^{2h+3}u 
\, u^{p-1}\, \bar u^p\\
- \mathrm{Im}\int\partial_x^{k-h}u\,\partial_x^{k-h-1} \bar u\, \partial_x^{2h+1}\bar u 
\, \bar u^{p-1}\, u^p- \mathrm{Im}\,\int \partial_x^{k-h-1} u\,\partial_x^{k-h-2} \bar u\, \partial_x^{2h+3}u
\, u^{p-1}\, \bar u^p
\end{multline*}
and we are done.\end{proof}
\begin{lemma}
\label{lem:3}
Let $k>2$ then we have:
\begin{equation}
\label{e:ast-to-worse22}
{\tilde{\mathcal K}_{k,m+1}^\ast}(u)\equiv 2\, \mathcal{K}_{k,m}(u),\quad 
{\tilde{\mathcal V}_{k,m+1}^\ast}(u)\equiv  2 \mathcal V_{k,m}.
\end{equation}
Moreover we have:
\begin{itemize}
\item for $k=3m$ then 
\begin{equation}\label{3m}
{\tilde{\mathcal I}_{k,m+1}^\ast}(u)\equiv 0,
\end{equation} 
and if $m>1$,
\begin{equation}\label{3m.2}
{\tilde{\mathcal W}_{k,m+1}^\ast}(u)\equiv  2{\mathcal V}_{k, m}
- 2{\mathcal W}_{k, m} - 2{\mathcal W}_{k, m-1}  - 2{\mathcal K}_{k, m-1}
+4{\mathcal K}_{k, m};\end{equation}
\item for $k=3m+1$ then
\begin{equation}\label{3m+1}
{\tilde{\mathcal I}_{k,m+1}^\ast}(u)\equiv 6 {\mathcal I}_{k,m};
\end{equation}
\item for $k=3m+2$ then
\begin{equation}\label{3m+2}
{\tilde{\mathcal I}_{k,m+1}^\ast}(u)\equiv 2 {\mathcal I}_{k,m}, \quad 
{\tilde{\mathcal W}_{k,m+1}^\ast}(u)\equiv 2 {\mathcal V}_{k,m} - {\mathcal W}_{k,m}
+ 2 {\mathcal K}_{k,m}.
\end{equation}
\end{itemize}

\end{lemma}
\begin{proof}
The proof of \eqref{e:ast-to-worse22} is identical to the one of the second and third equivalence in \eqref{e:ast-to-worse2} in the case $h=m$.
Next we focus on the remaining estimates.
First notice that by \eqref{eq:utildop} we have 
\begin{multline}\label{partfra}{\tilde{\mathcal I}_{k,m+1}^\ast}(u)
\equiv  2\, \mathrm{Im}\int \big(\partial_x^{k-m} u\big)^2\, \partial_x^{2m}u 
\, u^{p-2}\, \bar u^{p+1}\\- 2\, \mathrm{Im}\int (\partial_x^{k-m-1}u)^2\, \partial_x^{2m+2}u
\, u^{p-2}\, \bar u^{p+1}.
\end{multline}
Also by the last equivalence in \eqref{e:ast-to-worse2}, which is true also for $h=m$, we get
\begin{multline}\label{ritfra}
{\tilde{\mathcal W}_{k,m+1}^\ast}(u)\equiv
2 \mathrm{Im}\int\partial_x^{k-m}u\,\partial_x^{k-m-1} \bar u\, \partial_x^{2m+1}u 
\, u^{p-1}\, \bar u^p\\
-2\mathrm{Im}\int\partial_x^{k-m}u\,\partial_x^{k-m-1} \bar u\, \partial_x^{2m+1}\bar u 
\, \bar u^{p-1}\, u^p
\\-2\mathrm{Im}\int\partial_x^{k-m-1}u\,\partial_x^{k-m-2} \bar u\, \partial_x^{2m+3}u 
\, u^{p-1}\, \bar u^p.
\end{multline}
\\
{\bf (Case $k=3m$)}
In this specific case we get from \eqref{partfra}
\begin{multline*}
{\tilde{\mathcal I}_{k,m+1}^\ast}(u)
\equiv  2\, \mathrm{Im}\int \big(\partial_x^{2m} u\big)^3 
\, u^{p-2}\, \bar u^{p+1}- 2\, \mathrm{Im}\int (\partial_x^{2m-1}u)^2\, \partial_x^{2m+2}u
\, u^{p-2}\, \bar u^{p+1}\\
\equiv 2\, \mathrm{Im}\int \big(\partial_x^{2m} u\big)^3 
\, u^{p-2}\, \bar u^{p+1}+4\, \mathrm{Im}\int \partial_x^{2m-1}u \,  \partial_x^{2m}u \,
\partial_x^{2m+1}u\, u^{p-2}\, \bar u^{p+1}\\
\equiv 2\, \mathrm{Im}\int \big(\partial_x^{2m} u\big)^3 
\, u^{p-2}\, \bar u^{p+1}-2\, \mathrm{Im}\int (\partial_x^{2m}u)^3 \,u^{p-2}\, \bar u^{p+1}
\end{multline*}
and we conclude the equivalence in \eqref{3m}. Next notice that,  for $k=3m$ with $m>1$, we get from \eqref{ritfra}
\begin{multline}\label{prelW}
{\tilde{\mathcal W}_{k,m+1}^\ast}(u)
\equiv 
2 \mathrm{Im}\int\partial_x^{2m}u\,\partial_x^{2m-1} \bar u\, \partial_x^{2m+1}u 
\, u^{p-1}\, \bar u^p
\\
-2\mathrm{Im}\int\partial_x^{2m}u\,\partial_x^{2m-1} \bar u\, \partial_x^{2m+1}\bar u 
\, \bar u^{p-1}\, u^p
-2\mathrm{Im}\int\partial_x^{2m-1}u\,\partial_x^{2m-2} \bar u\, \partial_x^{2m+3}u 
\, u^{p-1}\, \bar u^p
\end{multline}
and we compute the last term on the r.h.s.
\begin{multline*}\mathrm{Im}\int\partial_x^{2m-1}u\,\partial_x^{2m-2} \bar u\, \partial_x^{2m+3}u 
\, u^{p-1}\, \bar u^p\equiv 
-\mathrm{Im}\int\partial_x^{2m}u\,\partial_x^{2m-2} \bar u\, \partial_x^{2m+2}u 
\, u^{p-1}\, \bar u^p
\\-\mathrm{Im}\int\partial_x^{2m-1}u\,\partial_x^{2m-1} \bar u\, \partial_x^{2m+2}u 
\, u^{p-1}\, \bar u^p\\
\equiv \mathrm{Im}\int\partial_x^{2m+1}u\,\partial_x^{2m-2} \bar u\, \partial_x^{2m+1}u 
\, u^{p-1}\, \bar u^p
+\mathrm{Im}\int\partial_x^{2m}u\,\partial_x^{2m-1} \bar u\, \partial_x^{2m+1}u 
\, u^{p-1}\, \bar u^p
\\+\mathrm{Im}\int\partial_x^{2m}u\,\partial_x^{2m-1} \bar u\, \partial_x^{2m+1}u 
\, u^{p-1}\, \bar u^p\
+\mathrm{Im}\int\partial_x^{2m-1}u\,\partial_x^{2m} \bar u\, \partial_x^{2m+1}u 
\, u^{p-1}\, \bar u^p\\
\equiv \mathrm{Im}\int\partial_x^{2m+1}u\,\partial_x^{2m-2} \bar u\, \partial_x^{2m+1}u 
\, u^{p-1}\, \bar u^p
+\mathrm{Im}\int\partial_x^{2m}u\,\partial_x^{2m-1} \bar u\, \partial_x^{2m+1}u 
\, u^{p-1}\, \bar u^p
\\+\mathrm{Im}\int\partial_x^{2m}u\,\partial_x^{2m-1} \bar u\, \partial_x^{2m+1}u 
\, u^{p-1}\, \bar u^p\
\\-\mathrm{Im}\int\partial_x^{2m}u\,\partial_x^{2m} \bar u\, \partial_x^{2m}u 
\, u^{p-1}\, \bar u^p
-\mathrm{Im}\int\partial_x^{2m-1}u\,\partial_x^{2m+1} \bar u\, \partial_x^{2m}u 
\, u^{p-1}\, \bar u^p
\\= {\mathcal K}_{k,m-1}-2{\mathcal K}_{k,m} +{\mathcal W}_{k,m-1}.\end{multline*}
The equivalence in \eqref{3m.2} then follows from  \eqref{prelW} and the equivalence above.
\\
\\
\noindent {\bf (Case $k=3m+1$)}
In this case we get from \eqref{partfra} the equivalence
\begin{multline*}{\tilde{\mathcal I}_{k,m+1}^\ast}(u)
\equiv  2\, \mathrm{Im}\int \big(\partial_x^{2m+1} u\big)^2\, \partial_x^{2m}u 
\, u^{p-2}\, \bar u^{p+1}\\- 2\, \mathrm{Im}\int (\partial_x^{2m}u)^2\, \partial_x^{2m+2}u
\, u^{p-2}\, \bar u^{p+1}\\
\equiv 
 2\, \mathrm{Im}\int \big(\partial_x^{2m+1} u\big)^2\, \partial_x^{2m}u 
\, u^{p-2}\, \bar u^{p+1}+4\, \mathrm{Im}\int \partial_x^{2m}u\, \partial_x^{2m+1}u\, 
\partial_x^{2m+1}u
\, u^{p-2}\, \bar u^{p+1}
\end{multline*}
and hence we conclude \eqref{3m+1}. 
\\
\\
{\bf (Case $k=3m+2$)}
First we notice that from \eqref{partfra}
\begin{multline*}{\tilde{\mathcal I}_{k,m+1}^\ast}(u)
\equiv  2\, \mathrm{Im}\int \big(\partial_x^{2m+2} u\big)^2\, \partial_x^{2m}u 
\, u^{p-2}\, \bar u^{p+1} \\- 2\, \mathrm{Im}\int (\partial_x^{2m+1}u)^2\, \partial_x^{2m+2}u
\, u^{p-2}\, \bar u^{p+1}\equiv 2\, \mathrm{Im}\int \big(\partial_x^{2m+2} u\big)^2\, \partial_x^{2m}u 
\, u^{p-2}\, \bar u^{p+1}
\end{multline*}
and we conclude the first equivalence in \eqref{3m+2}.
Next we compute by \eqref{ritfra}
\begin{multline*}
{\tilde{\mathcal W}_{k,m+1}^\ast}(u)\equiv
2 \mathrm{Im}\int\partial_x^{2m+2}u\,\partial_x^{2m+1} \bar u\, \partial_x^{2m+1}u 
\, u^{p-1}\, \bar u^p\\
-2\mathrm{Im}\int\partial_x^{2m+2}u\,\partial_x^{2m+1} \bar u\, \partial_x^{2m+1}\bar u 
\, \bar u^{p-1}\, u^p
-2\mathrm{Im}\int\partial_x^{2m+1}u\,\partial_x^{2m} \bar u\, \partial_x^{2m+3}u 
\, u^{p-1}\, \bar u^p
\\ \equiv
2 \mathrm{Im}\int\partial_x^{2m+2}u\,\partial_x^{2m+1} \bar u\, \partial_x^{2m+1}u 
\, u^{p-1}\, \bar u^p\\
-2\mathrm{Im}\int\partial_x^{2m+2}u\,\partial_x^{2m+1} \bar u\, \partial_x^{2m+1}\bar u 
\, \bar u^{p-1}\, u^p\\
+2\mathrm{Im}\int\partial_x^{2m+1}u\,\partial_x^{2m+1} \bar u\, \partial_x^{2m+2}u 
\, u^{p-1}\, \bar u^p
+2\mathrm{Im}\int\partial_x^{2m+2}u\,\partial_x^{2m} \bar u\, \partial_x^{2m+2}u 
\, u^{p-1}\, \bar u^p\\
\equiv 2 \mathrm{Im}\int\partial_x^{2m+2}u\,\partial_x^{2m+1} \bar u\, \partial_x^{2m+1}u 
\, u^{p-1}\, \bar u^p\\
-2\mathrm{Im}\int\partial_x^{2m+2}u\,\partial_x^{2m+1} \bar u\, \partial_x^{2m+1}\bar u 
\, \bar u^{p-1}\, u^p\\
-\mathrm{Im}\int (\partial_x^{2m+1}u)^2\,\partial_x^{2m+2} \bar u
\, u^{p-1}\, \bar u^p
+2\mathrm{Im}\int\partial_x^{2m+2}u\,\partial_x^{2m} \bar u\, \partial_x^{2m+2}u 
\, u^{p-1}\, \bar u^p\
\end{multline*}
and we conclude the second equivalence in \eqref{3m+2}.
\end{proof}
\noindent {\em Proof of Proposition \ref{prop:structure}.}
First notice that for any integer $k\geq 2$ and $h\in \{0,\dots, m\}$
\begin{equation}\label{starstar}\tilde {\mathcal I}_{k, h+1}^{\ast\ast}(u), 
\tilde {\mathcal J}_{k,h+1}^{\ast\ast}(u), 
\tilde {\mathcal V}_{k,h+1}^{\ast\ast}(u), 
\tilde {\mathcal W}_{k,h+1}^{\ast\ast}(u)\equiv 0.\end{equation}
Hence, by Lemma \ref{Hk}, we conclude by using \eqref{e:asts} once we  prove that there exist $\tilde \alpha_h, \tilde \beta_h, \tilde \gamma_h$, $h\in \{0,\dots, m\}$ and $\tilde \delta_h,\, h\in \{1,\dots, m\}$ such that
\begin{multline}\label{saT}\sum_{h=0}^{m} \left(\tilde{\alpha}_h\, \tilde{\mathcal I}_{k,h+1}^\ast(u)  + \, \tilde{\beta}_h\, \tilde{\mathcal V}_{k,h+1}^\ast (u)+\tilde{\gamma}_h\, \tilde{\mathcal W}_{k,h+1}^\ast(u)\right)
+\sum_{h=1}^{m} \tilde{\delta}_h\, \tilde{\mathcal K}_{k,h+1}^\ast (u)\\
\equiv - \sum_{h=0}^{m}\Big(\alpha_h\, \mathcal I_{k,h} (u)+\, \beta_h\, \mathcal V_{k,h}(u)+ \gamma_h\, \mathcal W_{k,h}(u)\Big)
\\- \sum_{h=1}^m \delta_h\, \mathcal K_{k,h} (u) \hbox{ for } k=3m+1,\, 3m+2,\end{multline}
and
\begin{multline}\label{saT2}\sum_{h=0}^{m} \left(\tilde{\alpha}_h\, \tilde{\mathcal I}_{k,h+1}^\ast(u) +\tilde{\beta}_h\, \tilde{\mathcal V}_{k,h+1}^\ast(u)+\tilde{\gamma}_h\, \tilde{\mathcal W}_{k,h+1}^\ast (u)\right)+ \sum_{h=1}^m \tilde{\delta}_h\, \tilde{\mathcal K}_{k,h+1}^\ast (u) \\\equiv c \,
\mathrm{Im}\int \big(\partial_x^{2m} u\big)^3\, u^{p-2}\, \bar u^{p+1}   - \sum_{h=0}^{m}\Big(\alpha_h\, \mathcal I_{k,h}(u)+ \beta_h\, \mathcal V_{k,h}(u)+ \gamma_h\, \mathcal W_{k,h}(u)\Big)\\-\, \sum_{h=1}^m \delta_h\, \mathcal K_{k,h}(u)  \hbox{ for } k=3m,\end{multline}
where $\alpha_h, \beta_h, \gamma_h, \delta_h$
are the same constants in \eqref{e:der} and $c\in \R$ is a suitable constant. 

\smallskip
We preliminary treat the special and easier cases $k=2$ and $k=3$.\\
\\
{\bf First case: $k=2$.} 
We have to prove the existence of $\tilde \alpha_0, \tilde \beta_0, \tilde \gamma_0$ such that
\begin{multline}\label{conad}\tilde{\alpha}_0\, \tilde{\mathcal I}_{2,1}^\ast(u)  + \, \tilde{\beta}_0\, \tilde{\mathcal V}_{2,1}^\ast (u)+\tilde{\gamma}_0\, \tilde{\mathcal W}_{2,1}^\ast(u)\\
\equiv - \alpha_0\, \mathcal I_{2,0} (u)-\, \beta_0\, \mathcal V_{2,0}(u)- \gamma_0\, \mathcal W_{2,0}(u)\end{multline}
where $\alpha_0, \beta_0, \gamma_0$ are set, for $k=2$, in \eqref{e:der}.
Notice that by \eqref{e:corrections} and integration by parts we have the identities
$$\tilde{\mathcal I}_{2,1}=\tilde{\mathcal W}_{2,1}, \quad
\mathcal W_{2,0}(u) \equiv 2\,\mathcal V_{2,0}(u).
$$ Hence the condition \eqref{conad} can be written as follows:
\begin{equation*}(\tilde \alpha_0 + \tilde \gamma_0) {\tilde{\mathcal I}_{2,1}^\ast}(u)+ \, \tilde{\beta}_0\, \tilde{\mathcal V}_{2,1}^\ast (u)
\equiv - \alpha_0\, \mathcal I_{2,0} (u)-(\, \beta_0 + 2 \gamma_0)\, \mathcal V_{2,0}(u)
\end{equation*}
By Lemma \ref{lem:1} and by using  $ \mathcal I_{2,1}(u)\equiv 0$ (which in turn follows by \eqref{e:corrections} and integration by parts) we get
$${\tilde{\mathcal I}_{2,1}^\ast}(u)\equiv 2\, \mathcal I_{2,0}(u)-2(p-1)\, \mathcal I_{2,1}(u)
\equiv 2\, \mathcal I_{2,0}(u),\quad 
{\tilde{\mathcal V}_{2,1}^\ast}(u)\equiv  4p\, \mathcal V_{2,0}(u)$$
and hence the condition \eqref{conad} holds if we impose:
$$
\begin{cases}
2 (\tilde \alpha_0 +  \tilde \gamma_0)=-\alpha_0\\
4p \tilde \beta_0 =-\beta_0 - 2 \gamma_0.
\end{cases}
$$
Hence we get a $2\times 2$ linear system on the unknowns $\tilde \alpha_0 +  \tilde \gamma_0, \tilde \beta_0$ whose
associated matrix is invertible
$$
 \begin{bmatrix} 2 & 0 \\ 0& 4p\end{bmatrix}.
$$
Notice that we can select uniquely $\tilde \alpha_0 +  \tilde \gamma_0$ and $\tilde \beta_0$, and hence we have infinitely many choices 
for  $\tilde \alpha_0$ and $\tilde \gamma_0$.
\\
\noindent {\bf Second case: $k=3$.} In this case by \eqref{e:corrections} and integrating by parts we get
\begin{multline*}
{\tilde{\mathcal W}_{3,2}}(u)\equiv -\mathrm{Re}\int \big(\partial_x^2 u\big)^2\, u^{p-1}\, \bar u^{p+1} -  (p+1) \mathrm{Re} \int \partial_x^2 u\, \partial_x \bar u\, \partial_x u\, u^{p-1}\, \bar u^p
\\
 = - \tilde{\mathcal I}_{3,1}(u)-(p+1) \tilde{\mathcal W}_{3,1}(u).
\end{multline*}
On the other hand we have
\begin{multline*}
\mathcal W_{3,1}(u)=\mathrm{Im}\int \partial_x^2 u\, \partial_x \bar u\, \partial_x^3 \bar u\, u^p\, \bar u^{p-1} = -\mathrm{Im}\int \partial_x^2 \bar u\, \partial_x u\, \partial_x^3 u\, u^{p-1}\, \bar u^{p} = -\mathcal V_{3,0}(u).
\end{multline*}
From Lemmas \ref{lem:1}, \ref{lem:2}, and \ref{lem:3}, we also have
\begin{multline*}
{\tilde{\mathcal I}_{3,1}}^\ast(u)  \equiv 2\, \mathcal I_{3,0}(u)-2(p-1)\, \mathcal I_{3,1}(u),\, {\tilde{\mathcal I}_{3,2}}^\ast(u) \equiv 0, \, {\tilde{\mathcal V}_{3,1}}^\ast(u) \equiv 4\, p\, \mathcal V_{3,0}(u)\\ {\tilde{\mathcal V}_{3,2}}^\ast(u) \equiv 2\, \mathcal V_{3,1}(u), \, {\tilde{\mathcal W}_{3,1}}^\ast(u) \equiv 2\, \mathcal V_{3,0}(u)-2\, \mathcal W_{3,0}(u)-2\, \mathcal V_{3,1}(u), \, {\tilde{\mathcal K}_{3,2}}^\ast(u) \equiv 2\, \mathcal K_{3,1}(u).
\end{multline*}
Hence \eqref{saT2} follows if we can find 
$\tilde \alpha_0, \tilde \beta_0, \tilde \gamma_0,  \tilde \beta_1, \tilde \gamma_1, \tilde \delta_1,$ such that
\begin{multline}
\label{e:keq3}
(\tilde \alpha_0 - \tilde \gamma_1 )\, {\tilde{\mathcal I}_{3,1}}^\ast(u) + \tilde \beta_0  \, {\tilde{\mathcal V}_{3,1}}^\ast(u) + (\tilde \gamma_0 -(p+1)\, \tilde \gamma_1)\, {\tilde{\mathcal W}_{3,1}}^\ast(u) 
\\
+ \tilde \beta_1 \, {\tilde{\mathcal V}_{3,2}}^\ast(u) + \tilde \delta_1\, {\tilde{\mathcal K}_{3,2}}^\ast(u)
\equiv c \, \mathcal I_{3,1}(u) -\alpha_0 \, \mathcal I_{3,0}(u)-(\beta_0-\gamma_1)\, \mathcal V_{3,0}(u) 
\\
-\gamma_0\, \mathcal W_{3,0}(u) -\beta_1\, \mathcal V_{3,1}(u)- \delta_1\, \mathcal K_{3,1}(u),
\end{multline}
which in turn by the relations above is satisfied provided that we impose the system
$$
\begin{cases}
2\, \tilde \alpha_0 -2\, \tilde \gamma_1=-\alpha_0
\\
4\, p\, \tilde \beta_0 - 2(p+1)\, \tilde \gamma_1 + 2\, \tilde \gamma_0=-\beta_0+\gamma_1
\\
2\, \tilde \beta_1-2\, \tilde \gamma_0 +2 (p+1) \tilde \gamma_1=-\beta_1
\\
-2\, \tilde \gamma_0 + 2 (p+1) \tilde \gamma_1=-\gamma_0
\\
2\tilde \delta_1=-\delta_1.
\end{cases}
$$
Notice that we have a system of five equations with six variables, hence we can fix for instance $\tilde \gamma_1=0$ and the reduced corresponding linear system
 is associated with the matrix
$$
 \begin{bmatrix} 2 & 0  & 0 & 0 & 0\\ 0& 4p & 2 & 0 & 0\\
 0 & 0 & -2 & 2 & 0\\
 0 & 0 & -2 & 0 & 0\\
 0 & 0 & 0 & 0 & 2\end{bmatrix}
$$  
which is invertible.

\smallskip
Next we consider separately the three cases $k=3m+1$, $k=3m+2$ and $k=3m$, with $m>0$.
 \\
 \\
 \noindent {\bf Third case: $k=3m+1$, $m>0$.}    
 We shall prove the following facts, which in turn imply \eqref{saT} for $k=3m+1$:
\begin{itemize}
\item there exist $\tilde \alpha_h\in \mathbb{R}$, $h=0,\ldots,m$ such that 
\begin{equation}
\label{e:case1}
\sum_{h=0}^{m}\tilde \alpha_h\, \tilde{\mathcal I}_{k,h+1}^\ast(u) \equiv - \sum_{h=0}^{m} \alpha_h\, \mathcal I_{k,h}(u);
\end{equation}
\item there exist $\tilde \beta_h, \tilde \gamma_h, \beta \in \mathbb{R}$, $h=0,\ldots, m$ such that
\begin{multline}
\label{e:case3p}
\sum_{h=0}^{m}\Big(\tilde \beta_{h}\, \tilde{\mathcal V}_{k,h+1}^\ast(u) + \tilde \gamma_{h}\, \tilde{\mathcal W}_{k,h+1}^\ast (u)\Big)\\\equiv - \sum_{h=0}^{m}\Big( \beta_h\, \mathcal V_{k,h}(u) + \gamma_h \, \mathcal W_{k,h}(u) \Big)+ \beta \,\mathcal K_{k,m}(u);
\end{multline}
\item there exist $\tilde \delta_h$, $h=1, \ldots, m$ such that
\begin{equation}
\label{e:case2}
\sum_{h=1}^{m}\tilde \delta_h\, \tilde{\mathcal K}_{k,h+1}^\ast(u) \equiv - \beta \, \mathcal K_{k,m}(u) - \sum_{h=1}^{m} \delta_h\, \mathcal K_{k,h}(u)
\end{equation}
where $\beta$ is the same constant that appears in \eqref{e:case3p}.
\end{itemize}
We begin with proving \eqref{e:case1}. 
From Lemmas \ref{lem:1}, \ref{lem:2}, and \ref{lem:3},
we get
\[
\left\{
\begin{aligned}
\tilde{\mathcal I}_{k,1}^\ast(u)&\equiv 2\, \mathcal I_{k,0}(u) -2(p-1)\, \mathcal I_{k,1}(u) 
\\
\vdots
\\
\tilde{\mathcal I}_{k,h+1}^\ast (u)&\equiv 2\, \mathcal I_{k,h}(u) -2\, \mathcal I_{k,h+1} (u)
\\
\vdots
\\
\tilde{\mathcal I}_{k, m+1}^\ast(u) &\equiv  6 \,\mathcal I_{k,m}(u).
\end{aligned}
\right.
\]
Thus \eqref{e:case1} is associated with as a linear system with matrix as follows:
\[A=\begin{bmatrix}
2 & 0 & 0 & \cdots & \cdots &  \cdots & 0\\
-2(p-1) & 2 & 0 & \ddots & \ddots &   \ddots & 0
\\
0 & -2 & 2 & 0 & \ddots & \ddots & 0
\\
0& 0 & -2 & 2 & 0& \ddots & 0 
\\
\vdots & \ddots & \ddots & \ddots & \ddots & \ddots &\vdots 
\\
0 & \ddots & \ddots & \ddots & -2 & 2 & 0 \\
0 &\cdots & \cdots & \cdots  & 0 & -2 & 6 
\end{bmatrix}
\]
The matrix is clearly invertible, since it is a lower triangular matrix, whose diagonal entries are non-zero
and hence \eqref{e:case1} follows.
\smallskip
We prove next \eqref{e:case3p}. To do so first notice that 
since we are assuming $k=2m+1$ then we get
\begin{equation}\label{rmoa}
\mathcal V_{k,m}(u) =  \mathrm{Im}\int \partial_x^{2m+1} u\, \partial_x^{2m}\bar u\,\partial_x^{2m+1} u\, u^{p-1}\, \bar u^{p}  = \mathcal K_{k,m}\end{equation}
and also
\begin{multline}\label{rmo}
\mathcal W_{k,m}(u)= \mathrm{Im}\int \partial_x^{2m+1} u\, \partial_x^{2m}\bar u\,\partial_x^{2m+1}\bar u\, u^{p}\, \bar u^{p-1}
\\
\equiv - \mathrm{Im}\int \partial_x^{2m+2} u\, \partial_x^{2m-1}\bar u\,\partial_x^{2m+1}\bar u\, u^{p}\, \bar u^{p-1}-
\mathrm{Im}\int \partial_x^{2m+1} u\, \partial_x^{2m-1}\bar u\,\partial_x^{2m+2}\bar u\, u^{p}\, \bar u^{p-1}
\\
=-\, \mathcal W_{k,m-1}(u)+\mathcal V_{k,m-1}(u) .
\end{multline}
Hence we get that \eqref{e:case3p} follows if we show that we can select $
\tilde \beta_h, \tilde \gamma_h$ and $\beta'$ such that
\begin{multline}
\label{e:case3pp}
\sum_{h=0}^{m}\Big(\tilde \beta_{h}\, \tilde{\mathcal V}_{k,h+1}^\ast(u) + \tilde \gamma_{h}\, \tilde{\mathcal W}_{k,h+1}^\ast (u)\Big)\\\equiv - \sum_{h=0}^{m-1}\Big( \beta_h'\, \mathcal V_{k,h}(u) + \gamma_h' \, \mathcal W_{k,h}(u) \Big)+ \beta' \,\mathcal K_{k,m}(u),
\end{multline}
where the coefficients $\beta_h'$, $\gamma_h'$ are uniquely defined by $\beta_h, \gamma_h$
once we take into account \eqref{rmoa} and \eqref{rmo}.
Indeed we shall prove that 
we can select $
\tilde \beta_h, \tilde \gamma_h$ for $h=0, \dots, m-1$ such that
\begin{multline}
\label{e:case3ppp}
\sum_{h=0}^{m-1}\Big(\tilde \beta_{h}\, \tilde{\mathcal V}_{k,h+1}^\ast(u) + \tilde \gamma_{h}\, \tilde{\mathcal W}_{k,h+1}^\ast (u)\Big)\\\equiv - \sum_{h=0}^{m-1}\Big( \beta_h'\, \mathcal V_{k,h}(u) + \gamma_h' \, \mathcal W_{k,h}(u) \Big)+ \beta' \,\mathcal K_{k,m}(u).
\end{multline}
 Hence, we deal with the system given by $\tilde {\mathcal V}_{k,h}^\ast(u)$ and $\tilde {\mathcal W}_{k,h}^\ast(u)$ in terms of linear combinations of ${\mathcal V}_{k,h}(u)$ and ${\mathcal W}_{k,h}(u)$. 
By Lemmas \ref{lem:1}, \ref{lem:2} we have
\begin{equation*}
\label{e:systemVW}
\left\{
\begin{aligned}
\tilde{\mathcal V}_{k,1}^\ast(u)&\equiv 4p \, \mathcal V_{k,0}(u)
\\
\tilde{\mathcal W}_{k,1}^\ast(u)&\equiv 2\, \mathcal V_{k,0}(u)-2\, \mathcal W_{k,0} (u)-2\, \mathcal V_{k,1}(u)
\\
\vdots
\\
\tilde{\mathcal V}_{k,h+1}^\ast(u)&\equiv 2\, \mathcal V_{k,h} (u)
\\
\tilde{\mathcal W}_{k,h+1}^\ast(u)&\equiv 2\, \mathcal V_{k,h}(u)-2\, \mathcal W_{k,h}(u) -2\, \mathcal V_{k,h+1}(u)
\\
\vdots
\\
\tilde{\mathcal V}_{k,m}^\ast(u)&\equiv 2\, \mathcal V_{k,m-1} (u)
\\
\tilde{\mathcal W}_{k,m}^\ast(u)&\equiv 2 {\mathcal V}_{k,m-1}(u)-2 {\mathcal W}_{k,m-1}(u)
- 2 {\mathcal K}_{k,m}(u),
\end{aligned}
\right.
\end{equation*}
where we have used in the last equation \eqref{rmoa}.
Hence, \eqref{e:case3p} is proved if the following matrix  is invertible
\[
\begin{bmatrix}
4p & 2 & 0 & 0 & 0 &\cdots & \cdots & 0
\\
0 & -2 & 0 & 0 & 0 &\ddots & \ddots &   0
\\
0 & -2  & 2 & 2 & 0& \ddots & \ddots & 0
\\
0 & 0 & 0 & -2 &  0& \ddots  &  \ddots & 0
\\
0 & 0 & 0 & -2 &  2 & \ddots & \ddots  & 0 
\\
\vdots & & \ddots & \ddots &\ddots & \ddots & \ddots & \vdots 
\\ 
\\
0 & 0 & 0 &0 & 0 & \cdots   & 2 & 2
\\
0 & 0 & 0 &0 & 0 & \cdots   & 0 & -2 
\end{bmatrix}
= 
\begin{bmatrix}
\mathcal A & \mathbf 0 & \cdots & \mathbf 0 & \mathbf 0
\\
\ast & \mathcal B& & \mathbf 0& \mathbf 0
\\
& \ddots & \ddots & \vdots & \vdots
\\
\mathbf \ast  & \ast & & \mathcal B & \mathbf 0
\\
\mathbf \ast & \mathbf \ast  & & \ast& \mathcal B
\end{bmatrix}
\]
with
\[
\mathcal A= \begin{bmatrix} 4p &2 \\ 0&-2 \end{bmatrix},
\mathcal B= \begin{bmatrix} 2 &2 \\ 0&-2 \end{bmatrix}.
\]
This matrix is  invertible, since it is a  block lower triangular matrix, whose diagonal blocks are all invertible.
It remains to prove \eqref{e:case2}, for which we need to consider the system given by $\tilde {\mathcal K}_{k,h}(u)$ in terms of $\mathcal K_{k,h}(u)$, as deduce by Lemmas \ref{lem:2}, 
\ref{lem:3}. The associated matrix is diagonal with non-zero (diagonal) entries,  hence it is invertible and then \eqref{e:case2} holds true.
\medskip\\
\noindent {\bf Fourth case $k=3m+2$, $m>0$.} We shall prove
\eqref{e:case1}, \eqref{e:case3p}, \eqref{e:case2}.
By using Lemmas \ref{lem:1}, \ref{lem:2}, \ref{lem:3} we get
\[
\left\{
\begin{aligned}
\tilde{\mathcal I}_{k,1}^\ast(u)&\equiv 2\, \mathcal I_{k,0}(u) -2(p-1)\, \mathcal I_{k,1}(u) 
\\
\vdots
\\
\tilde{\mathcal I}_{k,h+1}^\ast (u)&\equiv 2\, \mathcal I_{k,h}(u) -2\, \mathcal I_{k,h+1} (u)
\\
\vdots
\\
\tilde{\mathcal I}_{k, m+1}^\ast(u) &\equiv  2 \,\mathcal I_{k,m}(u)
\end{aligned}
\right.
\]
and hence \eqref{e:case1} follows exactly as in the case $k=2m+1$.
Next notice that 
by Lemmas \ref{lem:1}, \ref{lem:2}, \ref{lem:3} we get
\begin{equation*}
\label{e:systemVW.2}
\left\{
\begin{aligned}
\tilde{\mathcal V}_{k,1}^\ast(u)&\equiv 4p \, \mathcal V_{k,0}(u)
\\
\tilde{\mathcal W}_{k,1}^\ast(u)&\equiv 2\, \mathcal V_{k,0}(u)-2\, \mathcal W_{k,0} (u)-2\, \mathcal V_{k,1}(u)
\\
\vdots
\\
\tilde{\mathcal V}_{k,h+1}^\ast(u)&\equiv 2\, \mathcal V_{k,h} (u)
\\
\tilde{\mathcal W}_{k,h+1}^\ast(u)&\equiv 2\, \mathcal V_{k,h}(u)-2\, \mathcal W_{k,h}(u) -2\, \mathcal V_{k,h+1}(u)
\\
\vdots
\\
\tilde{\mathcal V}_{k,m+1}^\ast(u)&\equiv 2\, \mathcal V_{k,m} (u)
\\
\tilde{\mathcal W}_{k,m+1}^\ast(u)&\equiv 2 {\mathcal V}_{k,m} - {\mathcal W}_{k,m}
+ 2 {\mathcal K}_{k,m}.
\end{aligned}
\right.
\end{equation*}
Also in this case it is easy to check that the associated matrix to the previous system is invertible and we get \eqref{e:case3p}. Finally \eqref{e:case2} follows
as in the case $k=2m+1$ since the corresponding system is diagonal with non zero entries
on the diagonal.
\medskip\\
\noindent {\bf Fifth case: $k=3m$, $m>0$.}
We split the proof of \eqref{saT2} in the following steps:
\begin{itemize}
\item there exist $\tilde \alpha_h\in \mathbb{R}$, $h=0,\ldots,m-1$ such that 
\begin{equation}
\label{e:case1rm}
\sum_{h=0}^{m-1}\tilde \alpha_h\, \tilde{\mathcal I}_{k,h+1}^\ast(u) \equiv 
 c \,
\mathrm{Im}\int \big(\partial_x^{2m} u\big)^3\, u^{p-2}\, \bar u^{p+1} - \sum_{h=0}^{m-1} \alpha_h\, \mathcal I_{k,h}(u);
\end{equation}
\item there exist $\tilde \beta_h, \tilde \gamma_h, \beta, \beta' \in \mathbb{R}$, $h=0,\ldots, m$ such that
\begin{multline}
\label{e:caserm}
\sum_{h=0}^{m}\Big(\tilde \beta_{h}\, \tilde{\mathcal V}_{k,h+1}^\ast(u) + \tilde \gamma_{h}\, \tilde{\mathcal W}_{k,h+1}^\ast (u)\Big)\equiv - \sum_{h=0}^{m}\Big( \beta_h\, \mathcal V_{k,h}(u) + \gamma_h \, \mathcal W_{k,h}(u) \Big)\\+ \beta \,\mathcal K_{k,m}(u)
+\beta' \,\mathcal K_{k,m-1}(u);
\end{multline}
\item there exist $\tilde \delta_h$, $h=1, \ldots, m$ such that
\begin{equation}
\label{e:case2rm}
\sum_{h=1}^{m}\tilde \delta_h\, \tilde{\mathcal K}_{k,h+1}^\ast(u) \equiv -\beta \, \mathcal K_{k,m}(u) 
-\beta' \,\mathcal K_{k,m-1}(u) - \sum_{h=1}^{m} \delta_h\, \mathcal K_{k,h}(u)
\end{equation}
where $\beta, \beta'$ are the same constant that appears in \eqref{e:caserm}.
\end{itemize}
Concerning \eqref{e:case1rm}, we have by Lemmas \ref{lem:1}, \ref{lem:2}
\[
\left\{
\begin{aligned}
\tilde{\mathcal I}_{k,1}^\ast(u)&\equiv 2\, \mathcal I_{k,0}(u) -2(p-1)\, \mathcal I_{k,1}(u) 
\\
\vdots
\\
\tilde{\mathcal I}_{k,h+1}^\ast (u)&\equiv 2\, \mathcal I_{k,h}(u) -2\, \mathcal I_{k,h+1} (u)
\\
\vdots
\\
\tilde{\mathcal I}_{k,m}^\ast (u)&\equiv 2\, \mathcal I_{k,m-1}(u) -2\, \mathcal I_{k,m} (u)
\end{aligned}
\right.
\]
for $h=1,\cdots,m-2$. Notice that due to Lemma \ref{lem:3}
we cannot exploit $\tilde{\mathcal I}_{k,m+1}^\ast (u)$ to cancel the term $
{\mathcal I}_{k,m}(u)=\mathrm{Im}\int \big(\partial_x^{2m} u\big)^3\, u^{p-2}\, \bar u^{p+1}$.
Hence \eqref{e:case1rm} follows from the invertibility of the matrix
\[A=\begin{bmatrix}
2 & 0 & 0 & \cdots & \cdots & \cdots & 0\\
-2(p-1) & 2 & 0 & \cdots & \cdots &   \cdots & 0
\\
0 & -2 & 2 & 0 & \cdots & \cdots & 0
\\
0& 0 & -2 & 2 & 0& \cdots & 0
\\
& \ddots & \ddots & \ddots & \ddots &\ddots 
\\
0 & \cdots & \cdots & \cdots & -2 & 2 & 0\\
0 &\cdots & \cdots & \cdots & 0 & -2 & 2
\end{bmatrix}
\]
which is lower triangular with non zero entries on the diagonal.
In order to prove \eqref{e:caserm} we use Lemmas \ref{lem:1}, \ref{lem:2}, \ref{lem:3} and we get
\begin{equation*}
\label{e:systemVW.3}
\left\{
\begin{aligned}
\tilde{\mathcal V}_{k,1}^\ast(u)&\equiv 4p \, \mathcal V_{k,0}(u)
\\
\tilde{\mathcal W}_{k,1}^\ast(u)&\equiv 2\, \mathcal V_{k,0}(u)-2\, \mathcal W_{k,0} (u)-2\, \mathcal V_{k,1}(u)
\\
\vdots
\\
\tilde{\mathcal V}_{k,h+1}^\ast(u)&\equiv 2\, \mathcal V_{k,h} (u)
\\
\tilde{\mathcal W}_{k,h+1}^\ast(u)&\equiv 2\, \mathcal V_{k,h}(u)-2\, \mathcal W_{k,h}(u) -2\, \mathcal V_{k,h+1}(u)
\\
\vdots
\\
\tilde{\mathcal V}_{k,m+1}^\ast(u)&\equiv 2\, \mathcal V_{k,m} (u)
\\
\tilde{\mathcal W}_{k,m+1}^\ast(u)&\equiv  2{\mathcal V}_{k, m}
- 2{\mathcal W}_{k, m} - 2{\mathcal W}_{k, m-1}  - 2{\mathcal K}_{k, m-1}
+4{\mathcal K}_{k, m}.
\end{aligned}
\right.
\end{equation*}
Hence in order to deduce \eqref{e:caserm} we have to consider the invertibility of a block lower triangular matrix with the following structure
\[A=\begin{bmatrix}
4p & 2 & 0 & \cdots & \cdots & \cdots & \cdots & \cdots & 0\\
0 & -2 & 0 & \cdots & \cdots &   \cdots &  \cdots &  \cdots  & 0
\\
0 & -2 & 2 & 2 & \cdots & \cdots &  \cdots &  \cdots &0
\\
0& 0 & 0 & -2 & 0& \cdots  & \cdots   & \cdots  &0 
\\
& \ddots & \ddots & \ddots & \ddots &\ddots 
\\ 0 & \cdots & \cdots & \cdots  & -2 & 2 & 2& 0 & 0\\
0 &\cdots & \cdots & \cdots  & 0 & 0 & -2 & 0 & -2\\
0 &\cdots & \cdots & \cdots  & 0 & 0 & - 2 & 2 & 2\\
0 &\cdots & \cdots & \cdots  & 0 & 0 & 0 & 0 & -2
\end{bmatrix}
= \begin{bmatrix}
\mathcal A & \mathbf 0 & \cdots & \mathbf 0 & \mathbf 0
\\
\ast & \mathcal B& & \mathbf 0& \mathbf 0
\\
& \ddots & \ddots & \vdots & \vdots
\\
\mathbf \ast  & \ast & & \mathcal B & \mathbf 0
\\
\mathbf \ast & \mathbf \ast  & & \ast& \mathcal C
\end{bmatrix}
\]
with
\[
\mathcal A= \begin{bmatrix} 4p &2 \\ 0&-2 \end{bmatrix},
\mathcal B= \begin{bmatrix} 2 &2 \\ 0&-2 \end{bmatrix},
\mathcal C= \begin{bmatrix} 2 & 2 & 0 & 0 \\ 0 & -2 & 0 & -2\\ 0 & -2 & 2 & 2\\ 0&0&0&-2 \end{bmatrix}
\]
and clearly all the blocks are invertible by direct computation.
It remains to prove \eqref{e:case2rm}, for which we need to consider the system given by $\tilde {\mathcal K}_{k,h}(u)$ in terms of $\mathcal K_{k,h}(u)$, as deduced by Lemmas \ref{lem:2}, \ref{lem:3}. Since the associated matrix is diagonal with non-zero (diagonal) entries we conclude. 
\section{Proof of Theorem \ref{main}  for $k\neq 3m$}\label{diverso}
The energy ${\mathcal E}_k$ is provided in \eqref{defEk}. Next we shall use the notation
$${\mathcal E}_k(u)=\|u\|_{H^k}^2+{\mathcal F}_k(u),$$ 
where ${\mathcal F}_k(u)$
is a linear combination of terms of the following type
\begin{equation}\label{Ck}
\Im \int  \partial_x^{i_1} u\, \cdots \partial_x^{i_n}u \, \partial_x^{j_1}\bar u\cdots \partial_x^{j_n}\bar u, \quad (i,j)\in {\mathcal G}_k.\end{equation}
Moreover the r.h.s. of \eqref{e:sharp-estimate} is a linear combination of terms of the type: 
\begin{equation*}
\Re \int  \partial_x^{i_1} u\, \cdots \partial_x^{i_n}u \, \partial_x^{j_1}\bar u\cdots \partial_x^{j_n}\bar u, \quad (i,j)\in {\mathcal D}_k\cup {\mathcal C}_k.\end{equation*}
We claim that 
we have the bound 
\begin{equation}\label{insiem}
\int_{t_0}^{t_0+T}  \int \big| \partial_x^{i_1} u\, \cdots \partial_x^{i_n}u \, \partial_x^{j_1}\bar u\cdots \partial_x^{j_n}\bar u\big| \leq C \|u(t_0)\|_{H^k}^{\frac{2k-4}{k-1}}, \quad  (i,j)\in {\mathcal D}_k\end{equation}
provided that $T=T(k,R)>0$ is the one of Proposition \ref{cauchyprop}
and $u(t,x)$ solves \eqref{NLSt0}.
We also claim the following time-independent estimate
\begin{equation}\label{insiem*}
\int \big| \partial_x^{i_1} u\, \cdots \partial_x^{i_n}u \, \partial_x^{j_1}\bar u\cdots \partial_x^{j_n}\bar u\big| \leq C \|u\|_{H^k}^{\frac{2k-4}{k-1}}, \quad (i,j)\in {\mathcal C}_k\cup {\mathcal G}_k , \quad u\in H^k.\end{equation}
Once \eqref{insiem} and \eqref{insiem*} are established we conclude. In fact notice that 
integration on $(t_0, t_0+T)$ of the estimate \eqref{insiem*} along with Proposition \ref{cauchyprop} implies 
\begin{equation}\label{insiem11}
\int_{t_0}^{t_0+T}  \int \big| \partial_x^{i_1} u\, \cdots \partial_x^{i_n}u \, \partial_x^{j_1}\bar u\cdots \partial_x^{j_n}\bar u\big| \leq C \|u(t_0)\|_{H^k}^{\frac{2k-4}{k-1}}, \quad  (i,j)\in {\mathcal C}_k\cup  {\mathcal G}_k\end{equation}
provided that $u(t,x)$ solves \eqref{NLSt0}. Hence we get \eqref{2} as consequence of \eqref{insiem*} where we select
$(i,j)\in {\mathcal G}_k$ and we use Proposition \ref{cauchyprop}.
Similarly \eqref{3} (for $\varepsilon=0$) follows by \eqref{insiem} and \eqref{insiem11} where we choose
 $(i,j)\in {\mathcal C}_k$, once \eqref{e:sharp-estimate} is integrated in time on the interval $(t_0, t_0+T)$.
\\
We next focus on the proof of \eqref{insiem}. We can order the set $\{i_1,\dots, i_n, j_1, \dots , j_n\}$
in decreasing order
$$\alpha_1\geq \alpha_2\geq \dots \geq \alpha_{2n}, \quad 
\alpha_1, \alpha_2, \alpha_3, \alpha_4\geq 1, \quad \sum_{j=1}^{2n} \alpha_j=2k.$$
Hence by using the H\"older inequality and the Sobolev embedding 
$H^{1}\subset L^\infty$ 
we get
\begin{multline*}
\int_{t_0}^{t_0+T}\int  |\partial_x^{i_1} u\, \cdots \partial_x^{i_n}u \, \partial_x^{j_1}\bar u\cdots \partial_x^{j_n}\bar u|
\\\leq \prod_{i=1}^4 \|\partial_x^{\alpha_i} u\|_{L^4((t_0,t_0+T);L^4)} \times \prod_{j=5}^{2n} \|\partial_x^{\alpha_j} u\|_{L^\infty((t_0,t_0+T); L^\infty)}\\\leq C
\prod_{i=1}^4 \|u_0\|_{H^{\alpha_i}} \times \prod_{j=5}^{2n} \|u\|_{L^\infty((t_0,t_0+T); H^{ 1+\alpha_j})}\\\leq 
C  \|u(t_0)\|_{H^k}^{\sum_{i=1}^4 \frac{\alpha_i -1}{k-1}} \|u(t_0)\|_{H^1}^{\sum_{i=1}^4 \frac{k-\alpha_i}{k-1}}  \|u(t_0)\|_{H^k}^{\sum_{j=5}^{2n} \frac{\alpha_j}{k-1}} \|u(t_0)\|_{H^1}^{\sum_{j=5}^{2n} \frac{k-1-\alpha_j}{k-1}}  \\
\leq C \|u(t_0)\|_{H^k}^{\frac{2k-4}{k-1}}
\end{multline*}
where we have used interpolation at the last step along with the uniform control of $H^1$ (see \eqref{energyest}).\\
Next we prove \eqref{insiem*}. As for \eqref{insiem} we can 
order $\{i_1,\dots, i_n, j_1, \dots , j_n\}$
in decreasing order
$$\alpha_1\geq \alpha_2\geq \dots \geq \alpha_{2n}, \quad 0<\alpha_1, \alpha_2\leq k-1, 
\quad \sum_{j=1}^{2n} \alpha_j=2k-2.$$
Then by the H\"older inequality and the Sobolev embedding $H^1\subset L^\infty$
we can estimate 
\begin{multline*}
\int  |\partial_x^{i_1} u\, \cdots \partial_x^{i_n}u \, \partial_x^{j_1}\bar u\cdots \partial_x^{j_n}\bar u|  
\\\leq \prod_{i=1}^2 \|\partial_x^{\alpha_i} u\|_{L^2} \times \prod_{j=3}^{2n} \|\partial_x^{\alpha_j} u\|_{L^\infty}\leq C
\prod_{i=1}^2 \|u\|_{H^{\alpha_i}} \times \prod_{j=3}^{2n} \|u\|_{H^{1+\alpha_j }}\\\leq 
C  \|u\|_{H^k}^{\sum_{i=1}^2 \frac{\alpha_i -1}{k-1}} \|u\|_{H^1}^{\sum_{i=1}^2 \frac{k-\alpha_i}{k-1}}  \|u\|_{H^k}^{\sum_{j=3}^{2n} \frac{\alpha_j}{k-1}} \|u\|_{H^1}^{\sum_{j=3}^{2n} \frac{k-1-\alpha_j}{k-1}}  
\leq C  \|u\|_{H^k}^{\frac{2k-4}{k-1}}
\end{multline*}
where we have used interpolation at the last step along with the uniform control of $H^1$ (see \eqref{energyest}).
\section{Proof of Theorem \ref{main} for $k=3m$}

Notice that the case $k=3m$ is different from the case $k\neq 3m$ due to the first term that appears on the r.h.s. in \eqref{saT2}, hence our main task is to estimate
\begin{equation}\label{Immaginary}\Im \int_{t_0}^{t_0+T} \int \big(\partial_x^{2m} u\big)^3\, u^{p-2}\, \bar u^{p+1}.\end{equation}
All the other terms can be dealt with as we already did in section \ref{diverso}.\\
The key is the following proposition.
\begin{proposition}\label{bourgainspaces}
Let $\alpha,\beta\geq 0$ and $l\geq 2$ be integers. Then for every $\varepsilon>0$ there is $b<1/2$ such that for every $t_0$ we have:
\begin{multline}\label{bourgloc}
\Big|\int_{t_0}^{t_0+T} \int(\partial_x^l u)^3 u^\alpha \bar{u}^\beta\Big|\leq C
\|u\|_{X^{l-1+\varepsilon ,b}_{(t_0, t_0+T)}}\, \|u\|^2_{X^{l+\varepsilon ,b}_{(t_0, t_0+T)}}\|u\|^{\alpha+\beta}_{X^{1,b}_{(t_0, t_0+T)}}\\ \forall u \in X^{l+\varepsilon ,b}_{(t_0, t_0+T)}.
\end{multline}
\end{proposition}
We start with a result concerning Sobolev spaces on $\R$, rather useful along the proof of Proposition \ref{bourgainspaces}.
\begin{lemma}\label{lyON} Let $-\infty<a<a'<\infty$ and $b\in [0, \frac 12)$ then there exists $C>0$ such that:
\begin{equation}\label{cutoff}\|\chi_{[a,a']} u\|_{\dot H^b(\R)}\leq C \|u\|_{\dot H^b(\R)}.
\end{equation}
\end{lemma}
\begin{proof}
We shall use the following equivalence of norms
\begin{equation}\label{equivalence}\|u\|_{\dot H^b(\R)}^2\sim \int_\R\int_\R \frac{|u(x)-u(y)|^2}{|x-y|^{1+2b}} dxdy, \quad b\in (0,1),
\end{equation}
as well as the following Hardy type inequality
\begin{equation}\label{hardy}\int_\R \frac{|u|^2}{|x|^{2b} }\leq C  \|u\|_{\dot H^b(\R)}^2, \quad b\in (0,\frac 12)\end{equation}
For a proof of \eqref{equivalence} and \eqref{hardy} see \cite{BCD}. 
In the sequel for simplicity we assume $a=0, a'=1$ and we denote $I=[0,1]$, $I^c=\R\setminus [0, 1]$.
We denote $\tilde u=\chi_{I} u$
then thanks to \eqref{equivalence} the square of the  l.h.s. in \eqref{cutoff}
is equivalent to
\begin{multline*}
\int \int_{I\times I}  \frac{|\tilde u(x)-\tilde u(y)|^2}{|x-y|^{1+2b}} dxdy
+\int_{I} (\int_{I^c}  \frac{|\tilde u(x)-\tilde u(y)|^2}{|x-y|^{1+2b}} dy)dx
+
\int_{I^c} (\int_{I}\frac{ |\tilde u(x)-\tilde u(y)|^2}{|x-y|^{1+2b}} dy) dx
\\=\int \int_{I\times I}  \frac{|u(x)-u(y)|^2}{|x-y|^{1+2b}} dxdy
+\int_{I} (\int_{I^c}  \frac{|u(x)|^2}{|x-y|^{1+2b}} dy)dx
+
\int_{I^c} (\int_{I}\frac{ |u(y)|^2}{|x-y|^{1+2b}} dy) dx\\=A_1+A_2+A_3,
\end{multline*}
where we used that $\tilde u(x)-\tilde u(y)=0$ for $(x, y)\in I^c\times I^c$.
Of course we have $$A_1 \leq \int \int_{\R\times \R}  \frac{|u(x)-u(y)|^2}{|x-y|^{1+2b}} dxdy\sim \|u\|_{\dot H^b(\R)}^2$$
where we used \eqref{equivalence}.
Concerning $A_2$ we have
\begin{multline*}A_2=\int_{I} |u(x)|^2 (\int_{I^c}  \frac{1}{|x-y|^{1+2b}} dy) dx
\leq \int_{\R} \frac{|u(x)|^2}{|x|^{2b}} dx +  \int_{\R} \frac{|u(x)|^2}{|1-x|^{2b}} dx
\\\leq C \|u\|_{\dot H^b(\R)}^2\end{multline*}
where we used \eqref{hardy}, the translation invariance of the $\dot H^b$ norm and the elementary bound 
$$\int_{I^c} \frac 1{|x-y|^{1+2b}} dy \leq \frac C{|x|^{2b}} + \frac C{|1-x|^{2b}}, \quad \forall x\in (0,1).$$
Concerning $A_3$ we have
$$A_3=
\int_{I^c} (\int_{[0,1]}\frac{ |u(y)|^2}{|x-y|^{1+2b}} dy) dx
= \int_{I} |u(y)|^2 (\int_{I^c} \frac 1{|x-y|^{1+2b}} dx) dy$$
and hence it is exactly $A_2$.
Hence we conclude \eqref{cutoff}.

\end{proof}
\begin{proof}[Proof of Proposition \ref{bourgainspaces}]
For simplicity we restrict to the case $t_0=0$, the general case is identical. 
We denote by $\chi_T$ the characteristic function of the interval $[0,T]$.
We claim that it is sufficient to show the following bound:
\begin{equation}\label{globRR}
\Big|\int_\R\int(\partial_x^l u)^3 u^\alpha \bar{u}^\beta\Big|\leq C
\|u\|_{X^{l-1+\varepsilon ,b}}\, \|u\|^2_{X^{l+\varepsilon ,b}}\|u\|^{\alpha+\beta}_{X^{1,b}}, \quad \forall u\in X^{l+\varepsilon, b}
\end{equation}
for some $b<\frac 12$ dependent on $\varepsilon>0$.
We first show how the localized estimate \eqref{bourgloc} follows from the estimate \eqref{globRR}. 
We define the sequences $u_{n,1}\in X^{l+\varepsilon ,b}$, $u_{n,2}\in X^{l-1+\varepsilon ,b}$, $u_{n,3}\in X^{1,b}$
such that:
\begin{equation}\label{restrict}u_{n, k}(t,x)=u(t,x) \hbox{ on } [0, T]\times \T, \quad k=1,2,3\end{equation}
and
\begin{multline}\label{extension}\|u_{n,1}\|_{X^{l+\varepsilon ,b}}\rightarrow \|u\|_{X^{l+\varepsilon ,b}_{(0,T)}},\\
\|u_{n,2}\|_{X^{l-1+\varepsilon ,b}}\rightarrow \|u\|_{X^{l-1+\varepsilon ,b}_{(0,T)}},
\|u_{n,3}\|_{X^{1,b}}\rightarrow \|u\|_{X^{1,b}_{(0,T)}}.\end{multline}
Next notice that by Lemma \ref{lyON} and the definition of the Bourgain spaces we have
\begin{multline}\label{corlem}
\|\chi_T u_{n,1}\|_{X^{l+\varepsilon ,b}}\leq C\|u_{n,1}\|_{X^{l+\varepsilon ,b}},\\ \|\chi_T u_{n,2}\|_{X^{l-1+\varepsilon ,b}}\leq C\|u_{n,2}\|_{X^{l-1+\varepsilon ,b}},
\|\chi_T u_{n,3}\|_{X^{1,b}}\leq C\|u_{n,3}\|_{X^{1,b}}.\end{multline}
On the other hand by \eqref{restrict} and \eqref{corlem} we have 
$\chi_T u=\chi_T u_{n,1}\in X^{l+\varepsilon ,b}$ hence we can plug in the estimate \eqref{globRR} the function $\chi_T u$. We conclude \eqref{globRR}
by using again the property $\chi_T u= \chi_T u_{n,k}$ for every $n\in \N, k=1,2,3$ in conjunction with
\eqref{extension} and \eqref{corlem}.
\\

Next we prove \eqref{globRR}. Let $v$ be defined by
$$
\hat{v}(\tau,n)=|\hat{u}(\tau,n)|, \quad \tau\in \R, n\in \Z .
$$
Using the Fourier transform we can write 
\begin{equation}\label{big_sum}
\Big|\int_\R\int(\partial_x^l u)^3 u^\alpha \bar{u}^\beta\Big|\leq C
\int_{\Lambda}\, \sum_{\Theta}\prod_{j=1}^3|n_j|^l\, \hat{v}(\tau_j,n_j)\times\prod_{j=4}^{3+\alpha+\beta}\, \hat{v}(\tau_j,n_j),
\end{equation}
where 
$$
\Lambda=\Big\{(\tau_1,\cdots,\tau_{3+\alpha+\beta})\in\R^{3+\alpha+\beta}\,\colon\,\sum_{j=1}^{3+\alpha}\tau_j-\sum_{j=4+\alpha}^{3+\alpha+\beta}\tau_j=0\Big\}
$$
and
$$
\Theta=\Big\{(n_1,\cdots,n_{3+\alpha+\beta})\in\Z^{3+\alpha+\beta}\,\colon\,\sum_{j=1}^{3+\alpha}n_j-\sum_{j=4+\alpha}^{3+\alpha+\beta}n_j=0\Big\}.
$$
We evaluate the sum over $\Theta$ by sums on which $|n_j|$ are restricted to dyadic intervals $[N_j, 2N_j]$. We set
$$
N_{(4)}=\max(N_4,\cdots N_{3+\alpha+\beta}).
$$ 
We consider two case. 
We first assume $N_1\gg N_{(4)}$. In this case, we use that for
$$
(\tau_1,\cdots,\tau_{3+\alpha+\beta})\in\Lambda
$$
one has
$$
\Big|\sum_{j=1}^{3+\alpha}(\tau_j+n_j^2)-\sum_{j=4+\alpha}^{3+\alpha+\beta}(\tau_j+n_j^2)\Big|\geq 
n_1^2+n_2^2+n_3^2-\sum_{j=4}^{3+\alpha+\beta}n_j^2\geq C n_1^2. 
$$
We therefore have that 
\begin{equation}\label{vitto}
\max_{1\leq j\leq 3+\alpha+\beta}\, \, \langle \tau_j+n_j^2\rangle \geq C n_1^2.
\end{equation}
We now consider two sub cases of the first case. Let us first suppose that the $\max$ in the left hand-side of \eqref{vitto} is attained for a $j$ in $\{1,2,3\}$. 
Without restriction of the generality we can suppose that the $\max$ in the left hand-side of \eqref{vitto} is attained for $j=1$. 
Define $w_1$ and $w_2$ by
$$
\widehat{w_1}(\tau,n)=
|n|^{l-1+\varepsilon}\,  \langle \tau+n^2\rangle^b\hat{v}(\tau,n),\quad
\widehat{w_2}(\tau,n)=
|n|^{l}\, \hat{v}(\tau,n) .
$$
The contribution of the case under consideration to \eqref{big_sum} can be evaluated by
\begin{multline*}
\int_\R\int  w_1 \,w_2^2\, v^\alpha \bar{v}^\beta
\leq 
\|w_1\|_{L^2(\R; L^2)}\|w_2\|_{L^4(\R; L^4)}^2\|v\|_{L^\infty(\R; L^\infty)}^{\alpha+\beta}
\\
\leq C
\|u\|_{X^{l-1+\varepsilon ,b}}\, \|u\|^2_{X^{l,b}}\|u\|^{\alpha+\beta}_{X^{1,b}}\,,
\end{multline*}
where we have used the Strichartz bounds
$$
\|u\|_{L^4(\R; L^4))}\leq C \|u\|_{X^{0,b}}
$$
and 
$$
\|u\|_{L^\infty(\R; L^\infty)}\leq C \|u\|_{X^{1,b}}\,,
$$
where $b<1/2$ is sufficiently close to $1/2$. Let us next suppose that the $\max$ in the left hand-side of \eqref{vitto} is attained for a $j$ in $\{4,\cdots, 3+\alpha+\beta\}$. 
Without restriction of the generality we can suppose that the $\max$ in the left hand-side of \eqref{vitto} is attained for $j=4$. 
Define $w_1, $$w_2$ and $w_3$ by
$$
\widehat{w_1}(\tau,n)=|n|^{l-1+\varepsilon}\,  \hat{v}(\tau,n), \quad\widehat{w_2}(\tau,n)=|n|^{l}\, \hat{v}(\tau,n) 
$$
and
$$
\widehat{w_3}(\tau,n)=
 \langle \tau+n^2\rangle^b\hat{v}(\tau,n).
$$
The contribution of the case under consideration to \eqref{big_sum} can be evaluated by
\begin{multline*}
\int_\R\int w_1 \,w_2^2 w_3\, v^{\alpha-1} \bar{v}^\beta
\\
\leq 
\|w_1\|_{L^6(\R; L^6))}\|w_2\|_{L^6(\R;L^6))}^2
\|w_3\|_{L^2(\R; L^2)}
\|v\|_{L^\infty(\R; L^\infty)}^{\alpha+\beta-1}
\\
\leq C
\|u\|_{X^{l-1+\varepsilon ,b}}\, \|u\|^2_{X^{l,b}}\|u\|^{\alpha+\beta}_{X^{1,b}}\,,
\end{multline*}
where we have used the Strichartz bound
$$
\forall\, \varepsilon>0,\,\, \exists\,\, b<1/2\,\, :\,\,
\|u\|_{L^6(\R; L^6)}\leq C \|u\|_{X^{\varepsilon,b}}\,.
$$
Let us now consider the second case. Namely, we suppose that  $N_1\lesssim N_{(4)}$. Without restriction of the generality we can suppose that $N_{(4)}=N_4$ and we can write
\begin{multline*}
\prod_{j=1}^3|n_j|^l\, \hat{v}(\tau_j,n_j)\times \prod_{j=4}^{3+\alpha+\beta}\, \hat{v}(\tau_j,n_j)
\leq 
\\
 |n_1|^{l-1}\, \hat{v}(\tau_1,n_1)\times  \Big(\prod_{j=2}^3|n_j|^l\, \hat{v}(\tau_j,n_j)\Big)\times |n_4|\, \hat{v}(\tau_4,n_4)\times \prod_{j=5}^{3+\alpha+\beta}\, \hat{v}(\tau_j,n_j).
\end{multline*}
Define $w_1, $$w_2$ and $w_3$ by
$$
\widehat{w_1}(\tau,n)=|n|^{l-1}\,  \hat{v}(\tau,n), \quad\widehat{w_2}(\tau,n)=|n|^{l}\, \hat{v}(\tau,n) 
$$
and
$$
\widehat{w_3}(\tau,n)=
 |n|\hat{v}(\tau,n).
$$
The contribution of the case under consideration to \eqref{big_sum} can be evaluated by
\begin{multline*}
\int_\R\int w_1 \,w_2^2 w_3\, v^{\alpha-1} \bar{v}^\beta
\\
\leq 
\|w_1\|_{L^6(\R; L^6)}\|w_2\|_{L^6(\R; L^6)}^2
\|w_3\|_{L^2(\R; L^2)}
\|v\|_{L^\infty(\R; L^\infty)}^{\alpha+\beta-1}
\\
\leq C
\|u\|_{X^{l-1+\varepsilon ,b}}\, \|u\|^2_{X^{l+\varepsilon,b}}\|u\|^{\alpha+\beta}_{X^{1,b}}\,.
\end{multline*}
\end{proof}
\noindent {\em Proof of Theorem \ref{main} for $k=3m$}. We can now complete the proof of Theorem \ref{main} in the case $k=3m$.
By using Proposition \ref{bourgainspaces} for $l=2m$ and $\alpha=p-2$, $\beta=p+1$ we get the bound
\begin{multline}\label{bsxsb}
\big|\int_{t_0}^{t_0+T}\mathrm{Im}\int \big(\partial_x^{2m} u\big)^3\, u^{p-2}\, \bar u^{p+1}\big|
\\\leq C \|u(t_0)\|_{H^{2m-1+\varepsilon }}\, \|u(t_0)\|^2_{H^{2m+\varepsilon}} \|u(t_0)\|^{2p-1}_{H^{1}}\,,
\end{multline}
where $T$ is provided by Proposition \ref{cauchypropXsb}.
By interpolation we can continue the estimate above as follows:
\begin{equation*}\cdots \leq C \|u(t_0)\|_{H^{3m}}^\eta \, \|u(t_0)\|_{H^{3m}}^{2\theta} \|u(t_0)\|_{H^{1}}^{1-\eta}\, \|u(t_0)\|_{H^{1}}^{2(1-\theta)} \|u(t_0)\|^{2p-1}_{H^{1}}
\end{equation*}
where
$$2m-1+\varepsilon=3m \eta + 1-\eta, \; 2m+\varepsilon=3m \theta +1-\theta$$
and hence by the conservation of the energy we conclude
$$\cdots\leq C \|u(t_0)\|_{H^{3m}}^{\eta+2\theta}= C\|u(t_0)\|_{H^{3m}}^\frac{2m-2+\varepsilon+4m-2+2\varepsilon}{3m-1}$$
and we get Theorem \ref{main} since $k=3m$.\qed


\begin{thebibliography}{99}

\bibitem{Bo_gafa} J.~Bourgain, 
{\it Fourier transform restriction phenomena for certain lattice subsets and applications to nonlinear evolution equation}, Geom. and Funct. Anal. 3 (1993) 107--156, 209-262.

\bibitem{B} J.~Bourgain,  {\it On the growth in time of higher Sobolev norms of smooth solutions of Hamiltonian PDE}, Internat. Math. Res. Notices,  (1996) 6,  277--304.

\bibitem{B2} J.~Bourgain,   
{\it Remarks on stability and diffusion in high-dimensional Hamiltonian systems and partial differential equations},  Ergod. Th. \& Dynam. Sys.  24 (2004), 1331--1357.

\bibitem{BCD} H.~Bahouri, J.Y.~Chemin, R.~Danchin, {\it Fourier Analysis and Nonlinear Partial Differential Equations}, Springer. 

\bibitem{CKSTT2}
J.~Colliander, M.~Keel, G.~Staffilani, H.~Takaoka, T.~Tao,
{\it Transfer of energy to high frequencies in the cubic defocusing
              nonlinear {S}chr\"odinger equation}, Invent. Math.,
181 (2010) 1, 39--113.

\bibitem{CKO}  J.~Colliander, S.~Kwon, T. ~Oh, {\it A remark on normal forms and the ``upside-down'' {$I$}-method for periodic {NLS}: growth of higher {S}obolev norms}, J. Anal. Math., 118 (2012) 1, 55--82.

\bibitem{D} B.~Dodson, {\it Global well-posedness and scattering for the defocusing,  {$L^2$} critical, nonlinear {S}chr\"{o}dinger equation when  {$d=1$}}, Amer. J. Math., 138 (2016) 2, 531--569.

\bibitem{G}{M.~Guardia},
{\it Growth of {S}obolev norms in the cubic nonlinear
              {S}chr\"odinger equation with a convolution potential},
Comm. Math. Phys., 329 (2014) 1, 405--434.
     
\bibitem{GHP} M.~Guardia, E.~Haus, M.~Procesi,  {\it Growth of {S}obolev norms for the analytic {NLS} on {$\Bbb{T}^2$}}, Adv. Math., 301 (2016), 615--692.     
     
\bibitem{GK} M.~Guardia, V.~Kaloshin,
{\it Growth of {S}obolev norms in the cubic defocusing nonlinear {S}chr\"odinger equation}, J. Eur. Math. Soc., 17 (2015) 1, 71--149.

\bibitem{HPTV}  Z. Hani,  B. Pausader,  N. Tzvetkov,  N. Visciglia,
{\it Modified scattering for the cubic Schr\"odinger equation on product spaces and applications}, Forum Math. PI 3 (2015), e4, 63 pp.

\bibitem{HP} E.~Haus, M. Procesi, {\it Growth of {S}obolev norms for the quintic {NLS} on {$\T^2$}}, Anal. PDE, 8 (2015) 4, 883--922.

\bibitem{N} K.~Nakanishi,  {\it Energy scattering for nonlinear {K}lein-{G}ordon and
              {S}chr\"{o}dinger equations in spatial dimensions {$1$} and {$2$}},
    Journal of Functional Analysis,
     {169}
      (1999)
    1,
      201--225.
      
\bibitem{PTV_first} F. Planchon, N. Tzvetkov, N. Visciglia, {\it  On the growth of Sobolev norms for NLS on 2- and 3-dimensional manifolds},  Anal. PDE 10 (2017), 1123--1147. 

\bibitem{PTV} F.~Planchon, N.~Tzvetkov, N.~Visciglia, {\it Transport of Gaussian measures by the flow of the nonlinear Schr\" odinger equation},  Math. Ann. 378 (2020) 389--423.

\bibitem{PTV2} F.~Planchon, N.~Tzvetkov, N.~Visciglia, {\it Modified energies for the periodic generalized KdV equation and applications}, Ann. Inst. Henri Poincaré, Anal. Non Linéaire, 40 (2023) 863--917.

\bibitem{PTVHarm} F. Planchon,  N. Tzvetkov, N. Visciglia, {\it Growth of Sobolev norms for \(2 d\) NLS with harmonic potential. } Rev. Mat. Iberoam. 39 (2023) 1405--1436.

\bibitem{So1} V.~ Sohinger,
{\it Bounds on the growth of high {S}obolev norms of solutions to
              nonlinear {S}chr\"odinger equations on {$\Bbb R$}},
Indiana Univ. Math. J., 60 (2011) 5, 1487--1516.

\bibitem{So2}V.~Sohinger,
{\it Bounds on the growth of high {S}obolev norms of solutions to
nonlinear {S}chr\"odinger equations on {$S^1$}},
Differential Integral Equations, 24 (2011) 7-8, 653--718.

\bibitem{So3}V.~Sohinger,
{\it Bounds on the growth of high {S}obolev norms of solutions to
2{D} {H}artree equations},
Discrete Contin. Dyn. Syst., 32 (2012) 10, 3733--3771.
  
\bibitem{S} G.~Staffilani, {\it On the growth of high {S}obolev norms of solutions for {K}d{V}  and {S}chr\"odinger equations}, Duke Math. J., 86 (1997) 1, 109--142.

\bibitem{HT} H.~Takaoka, {\it On the growth of Sobolev norm for the cubic NLS on two dimensional product space},
Journal of Differential Equation, 394 (2024) 296-319. 

\bibitem{JT} J.~Thirouin, {\em On the Growth of Sobolev norms of Solutions of the fractional defocusing NLS equation on the circle}, to appear on Annales de l'Institut Henri Poincar\'e, Analyse non lin\'eaire 34(2) (2017) 509-531.

\bibitem{Z} S.~Zhong,
{\it The growth in time of higher {S}obolev norms of solutions to
              {S}chr\"odinger equations on compact {R}iemannian manifolds},
J. Differential Equations, 245 (2008) 2, 359--376.

\bibitem{Zy} A.~Zygmund, {\it Trigonometric Series}, Cambridge Mathematical Library, 3rd edition (2003).





\end{thebibliography}
\end{document}